\documentclass[10pt]{article}

\usepackage{a4wide}
\usepackage{amssymb,amsthm,amsmath}
\usepackage{mathrsfs} 
\usepackage{dsfont}   
\usepackage{xspace}
\usepackage{upgreek}
\usepackage{url}      
\usepackage{mathtools}
\usepackage{color}    
\usepackage{tikz}     
\usepackage{subfig}   
\usepackage{nccmath}   
\usepackage{enumerate} 

\usepackage{centernot} 

\renewcommand{\ae}{\text{a.e.}}
\newcommand{\AC}{\mathop{\mathrm{AC}}\!}
\newcommand{\sA}{\mathscr{A}}
\newcommand{\sB}{\mathscr{B}}

\newcommand{\BV}{\mathrm{BV}}
\newcommand{\sC}{\mathscr{C}}
\newcommand{\cadlag}{c\`adl\`ag }
\newcommand{\const}{\mathit{const}\,}
\newcommand{\e}{\mathrm{e}}
\newcommand{\EE}{\mathbb{E}}

\newcommand{\sF}{\mathcal{F}}
\newcommand{\sG}{\mathcal{G}}
\newcommand{\I}{\mathcal{I}}
\newcommand{\J}{\mathcal{J}}
\newcommand{\sK}{\mathscr{K}}

\mathchardef\mhyphen="2D 
\renewcommand{\emptyset}{\varnothing}

\newcommand{\NN}{\mathbb{N}}
\newcommand{\sO}{\mathscr{O}}

\newcommand{\PP}{\mathbb{P}}

\newcommand{\sQ}{\ensuremath{\mathcal{Q}}\xspace}

\newcommand{\sR}{\ensuremath{\mathcal{R}\xspace}}
\newcommand{\RR}{\ensuremath{\mathbb{R}}}
\renewcommand{\S}{\mathcal{S}}
\newcommand{\sS}{\mathscr{S}}
\newcommand{\II}{\mathds{1}}

\newcommand{\X}{\ensuremath{\mathcal{X}}\xspace}
\newcommand{\sY}{\ensuremath{\mathcal{Y}}\xspace}
\newcommand{\dd}{\mathrm{d}}
\newcommand{\norm}[1]{\ensuremath{\left\lVert{#1}\right\rVert}}
\newcommand{\abs}[1]{\ensuremath{\left\lvert{#1}\right\rvert}}
\def\hybridto{\;\substack{\rightharpoonup\\[-0.3em]\rightarrow}\;}
\newcommand{\one}{\mathds{1}}

\newcommand{\bk}{\bar{k}}
\newcommand{\TV}{\mathrm{TV}}

\newcommand{\epvar}{\mathrm{epvar}}

\newcommand{\super}[1]{^{\scriptscriptstyle{(#1)}}}

\DeclareMathOperator\Lip{Lip}
\DeclareMathOperator\Prob{Prob}

\newtheorem{theorem}{Theorem}[section]
\newtheorem{lemma}[theorem]{Lemma}
\newtheorem{proposition}[theorem]{Proposition}
\newtheorem{corollary}[theorem]{Corollary}
\newtheorem{definition}[theorem]{Definition}
\newtheorem{assumption}[theorem]{Assumption}
\newenvironment{remark}
  {\par\medbreak\refstepcounter{theorem}%
    \noindent\textbf{Remark~\thetheorem. }}%
  {\qed\par\medskip}

\numberwithin{equation}{section}

\makeatletter
  \newcommand{\eqnum}{\leavevmode\hfill\refstepcounter{equation}\textup{\tagform@{\theequation}}} 
\makeatother

\author{Robert I. A. Patterson and D. R. Michiel Renger}
\title{Large deviations of reaction fluxes}
\date{\today}

\begin{document}
\maketitle



\abstract{We study a system of interacting particles that randomly react to form new particles. The reaction flux is the rescaled number of reactions that take place in a time interval. We prove a dynamic large-deviation principle for the reaction fluxes under general assumptions that include mass-action kinetics. This result immediately implies the dynamic large deviations for the empirical concentration.
}

\section{Introduction}
\label{sec:introduction}

Since Boltzmann's microscopic interpretation of entropy it is clear that thermodynamics is inherently related to large deviations. Onsager, in his papers \cite{Onsager1931I,Onsager1931II} was able to extend this principle to the non-static regime - at least for reversible systems and close to equilibrium. More recently, it was shown that reversible stochastic particle systems induce a thermodynamically consistent gradient flow through their dynamical large deviations, see~\cite{Adams2011,MielkePeletierRenger2014}, and in particular~\cite{MielkePattersonPeletierRenger2015} for an application to chemical reactions. This characterises dynamic behavior even far from equilibrium. However, a thermodynamically consistent representation of non-reversible particle systems remains one of the main open problems of non-equilibrium thermodynamics.

The difficulty in understanding irreversible particle systems lies in the occurrence of non-trivial fluxes, which is why flux large deviations are a commonly studied object, see 
\cite{DerridaDoucotRoche2004,BodineauLagouge2010,BodineauLagouge2012,Derrida2007lecturenotes,Bertinietal2005,Bertinietal2006,BaiesiMaesNetocny2009} for examples covering Brownian motionss, random walkers and exclusion processes.
In this work we apply the flux approach to reacting particles on a discrete state space.

\paragraph{Reacting particle system.} We study a general network of reactions,
\begin{equation}
  \sum_{y\in\sY}\alpha\super{r}_y y \xrightarrow{\bk\super{r}} \sum_{y\in\sY} \beta\super{r}_y y, \qquad r\in\sR,
\label{eq:reaction network}
\end{equation}
where $\sY$ is finite set of species, and $\sR$ is a finite set of reactions, and $\bk\super{r}$ are the corresponding reaction rates. A typical choice of reaction rates is $\bk\super{r}(c)=\const\times \prod_{y\in\sY} c_y^{\alpha_y}$; this is called \emph{mass-action kinetics}, but we will consider a much more general class of rates. 

For example, one could have the reactions
\begin{align*}
  2\mathsf{H}_2 + \mathsf{O}_2 \xrightarrow{\bk\super{\mathrm{fw}}} 2 \mathsf{H}_2 \mathsf{O}, && \text{and} &&
  2\mathsf{H}_2 \mathsf{O} \xrightarrow{\bk\super{\mathrm{bw}}} 2\mathsf{H}_2 + \mathsf{O}_2.
\end{align*}
In this case the set of species is $\sY=\{\mathsf{H}_2,\mathsf{O}_2,\mathsf{H}_2 \mathsf{O}\}$, the set of reactions is $\sR=\{\mathrm{fw},\mathrm{bw}\}$, and $\bk\super{\mathrm{fw}},\bk\super{\mathrm{bw}}$ are the reaction rates that depend on the concentration of the species in $\sY$. Furthermore, the species needed for the reactions can be grouped in the vectors $\alpha\super{\mathrm{fw}}, \alpha\super{\mathrm{bw}}=(2,1,0),(0,0,2)$, and similarly for the species resulting from the reactions $\beta\super{\mathrm{fw}}, \beta\super{\mathrm{bw}}=(0,0,2),(2,1,0)$. These vectors are called \emph{complexes} or \emph{stoichiometric} coefficients, the latter being Greek for ``element counting''.

The reaction networks described above are commonly modelled by the following microscopic particle system, see the survey \cite{AndersonKurtz2011} and the references therein. If at some given time $t$ there are $N(t)$ particles of types $Y_1(t),\hdots, Y_{N(t)}(t)$ in the system with fixed volume $V$, then the empirical measure (or concentration) is defined as $C\super{V}(t):=V^{-1}\sum_{i=1}^{N(t)}\one_{Y_i(t)}$. With jump rate $k\super{r,V}(C\super{V}(t))$, also called \emph{propensity}, a reaction $r$ occurs, causing the concentration to jump to the new state $C\super{V}(t)+\frac1V\gamma\super{r}$, where $\gamma\super{r}=\beta\super{r}-\alpha\super{r}\in \RR^\sY$ is the \emph{effective stoichiometric vector} (sometimes called \emph{state change vector}) for reaction $r$ and these are collected in a matrix $\Gamma:=\lbrack \gamma\super{1} \hdots, \gamma\super{R} \rbrack$, which therefore maps rescaled reaction counts to changes in concentration.
Since the propensities $k\super{r,V}$ depend on the particles through the empirical concentration only, $C\super{V}(t)$ is a Markov jump process in $\RR^\sY$.
The volume $V$ controls the order of the (changing) number of particles in the system.

A classic result~\cite{Kurtz1970,Kurtz1972} says that the empirical measure $C\super{V}(t)$ converges as $V\to\infty$ to the solution of the \emph{reaction rate equation} $\dot c(t)=\sum_{r\in\sR}\bk\super{r}\big(c(t)\big)$, where $V^{-1} k\super{V,r}\to\bk\super{r}$ (in a way that we specify later).

\paragraph{Reaction Fluxes.}

More information is included in the \emph{integrated empirical reaction flux},
\begin{align*}
  W\super{V,r}(t) &=\tfrac1V\#\big\{\text{reactions } r \text{ that occurred in time }(0,t\rbrack \big\}.
\end{align*}
The pair $\big(C\super{V}(t),W\super{V}(t)\big)$ is then a Markov process in $\RR^\sY_+\times\RR^\sR_+$ with generator
\begin{equation}
  (\sQ\super{V} f)(c,w) = \sum_{r\in \sR} k\super{r,V}(c)\big( f(c+\tfrac1V\gamma\super{r},w+\tfrac1V\mathds1_r)-f(c,w)\big).
\label{eq:generator}
\end{equation}

As in the Kurtz limit, this pair converges to the solution of the system of ODEs
\begin{align*}
\begin{cases}
  \dot c(t) = \Gamma\dot w(t)=\sum_{r\in\sR} \dot w\super{r}(t)\gamma\super{r},\\
  \dot w(t) = \bk\big(c(t)\big).
\end{cases}
\end{align*}
The first equation is a continuity equation, which also holds almost surely for the microscopic pair $(C\super{V},W\super{V})$, for finite $V$.

\paragraph{Large deviations.} The dynamic large-deviation principle for the concentrations $C\super{V}$ have been proven in \cite{Feng1994dynamic,Leonard1995,DjehicheKaj1995,ShwartzWeiss1995,ShwartzWeiss2005,DupuisEllisWeiss1991,LiLin2015,DupuisRamananWu2016} under various assumptions. Large deviations for the pair $(C\super{V},W\super{V})$ of concentrations and fluxes is, as far as we are aware, a relatively untred area. Formal large-deviation calculations for the reaction fluxes are found in \cite{BaiesiMaesNetocny2009}, a rigorous proof for the independent case was given in \cite{Renger2017}, and a semigroup-based rigorous proof for a more general class of reaction fluxes can be found in \cite{Kraaij2017}, still excluding mass-action kinetics. In our main result, we prove a dynamical large-deviation principle for the process $(C\super{V},W\super{V})$, under initial distribution $(\mu\super{V},\delta_0)$, where we shall assume that $\mu\super{V}$ satisfies a large-deviation principle with some rate functional $\I_0$. The precise statement reads:

\begin{theorem}\label{th:ldp}
Let $\mu\super{V}$ satisfy a large-deviation principle with rate function $\I_0$, and let Assumptions~\ref{ass:cont initial ldp} on $\mu\super{V}$ and Assumption~\ref{ass:rrate conditions} on $k,\bk$ hold. Then the process $(C\super{V}(t),W\super{V}(t))_{t=0}^T$ satisfies a large-deviation principle in $\BV(0,T;\RR^\sY\times\RR^\sR_+)$, equipped with the hybrid topology, with good rate functional $\I_0\big(c(0)\big) + \J(c,w)$, where
\begin{align}
  &\J(c,w) := \begin{cases}
                \int_0^T\!\S\big(\dot w(t) \mid \bk(c(t)\big)\,\dd t, &(c,w)\in W^{1,1}\big(0,T;\RR^\sY\times\RR^\sR_+\big),
                \text{ and } \dot c = \Gamma\dot w,\\
                \infty,                                              &\text{otherwise},
              \end{cases}
\label{eq:LDP rate}
\end{align}
with relative entropy
\begin{align*}
  \S( j\mid\hat j)&:=
   \begin{cases}\sum_{r\in \sR} s( j\super{r}\mid\hat j\super{r}),  & \text{ if }  j \ll \hat j,\\
                +\infty & \text{ otherwise}, \qquad\text{and}
     \end{cases}\\
  s(j\super{r}\mid\hat j\super{r})&:=
    \begin{cases}
      j\super{r} \log\big( \frac{j\super{r}}{\hat j\super{r}} \big) -j\super{r} +\hat j\super{r}, &j\super{r}>0,\\
      \hat j,  &j\super{r}=0,
    \end{cases}
\end{align*}
where $j \ll \hat j$ means that for all $r \in \sR$ one has $\hat j\super{r} = 0 \implies j\super{r}=0$.
\end{theorem}
The precise set of assumptions will be stated in Section~\ref{subsec:assumptions}. We choose to work in the hybrid topology on the space of paths of bounded variation rather than the commonly used Skorohod topology since it is in some sense natural for jump processes, and the compactness criteria are very simple; we will introduce and comment on this space, topology and $\sigma$-algebra in more detail in Section~\ref{subsec:BV}.

As an immediate consequence of Theorem~\ref{th:ldp}, we obtain the large deviations for the concentrations:
\begin{corollary} Let $\mu\super{V}$ satisfy a large-deviation principle with rate function $\I_0$, and let Assumptions~\ref{ass:cont initial ldp} on $\mu\super{V}$ and Assumption~\ref{ass:rrate conditions} on $k,\bk$ hold. Then the process $C\super{V}$ satisfies a large-deviation principle in $\BV(0,T;\RR^\sY)$, equipped with the hybrid topology, with good rate functional $\I_0\big(c(0)\big) + \I(c)$, where
\begin{equation*}
  \I(c):=\inf_{\substack{w\in W^{1,1}(0,T;\RR^\sR_+):\\ \dot c = \Gamma \dot w}} \,\J(c,w).
\end{equation*}
\end{corollary}
Naturally, this result is consistent with the above mentioned articles, but now under a more general set of assumptions on the reaction rates. In particular, our assumptions allow for mass-action kinetics, as in~\cite{DupuisRamananWu2016}.

\paragraph{Initial conditions.} Throughout the paper we consider two different initial conditions. The main statement, Theorem~\ref{th:ldp} holds if the initial condition is random and satisfies a large-deviation principle. We will assume continuity of this initial large-deviation rate functional, which is essential to approximate the rate functional by sufficiently regular paths. For some results we shall consider a deterministic initial condition $C\super{V}(0)=\tilde c\super{V}(0)$ such that $c\super{V}(0)\to \tilde c(0) \in \RR^\sY$ for some limit initial condition. Those results can then be extended to random initial conditions via a mixture argument~\cite{Biggins2004}. For the integrated fluxes we set $W\super{V,r}(0)=0$ almost surely; we shall therefore always implicitly assume that any large-deviation rate blows up unless $w(0)=0$.

\paragraph{Strategy and overview.}

Section~\ref{sec:setting} describes the setting of the paper: the topology used for the dynamic large deviations, the precise assumptions on the propensities, reaction rates and initial condition. We then discuss existence and convergence of the path measures, which serves as a prerequisite for the large-deviations. Section~\ref{sec:analysis of rate} is dedicated to the analysis of the rate functional. Most importantly, it is shown that the rate functional has an alternative formulation as a convex dual, and that the rate functional can be approximated by curves that are sufficiently regular to be able to perform a change-of-measure. In a sense, these approximation lemmas are the core of the large-deviation proof. We shall see that the fact that the rate functional has a relatively simple formulation makes these proofs rather direct (which would be much more cumbersome when proving the large deviations of the concentrations only). Finally, Section~\ref{sec:ldp} is devoted to the proof of the large-deviation principle, Theorem~\ref{th:ldp}. It will be shown that one can always construct sufficiently steep compact cones on which the path measures place all but exponentially vanishing probability. We then show the lower bound of the measures with the random initial conditions via a double tilting argument, exploiting the approximation lemmas. After this, the upper bound is proven under deterministic initial conditions, which implies the large-deviations upper bound by a mixture argument.

\section{Setting}
\label{sec:setting}

In this section we specify the setting that we will be used in the paper. More specifically, we first introduce the hybrid topology used in the large deviations, and the precise assumptions on the propensities, reaction rates and initial condition that we will need. Finally, we construct the Markov process and its corresponding limit.

\subsection{The hybrid topology}
\label{subsec:BV}

For any path $(c,w)\in L^1(0,T;\RR^\sY\times\RR^\sR)$, the essential pointwise variation is
\begin{equation*}
  \epvar(c,w):=\inf_{\substack{(\tilde c, \tilde w) = (c, w)\\ t-\text{a.e.}}} \sup_{0=t_1<\hdots<t_K=T}\sum_{k=1}^K \big\lvert \big(\tilde c(t_{k+1}),\tilde w(t_{k+1})\big) - \big(\tilde c(t_{k}),\tilde w(t_{k})\big) \big\rvert,
\end{equation*}
and the space of paths of bounded variation is defined as:
\begin{equation*}
  \BV(0,T;\RR^\sY\times\RR^\sR):=\big\{ (c,w)\in L^1(0,T;\RR^\sY\times\RR^\sR) : \epvar(c,w)<\infty \big\}.
\end{equation*}
Some key properties of paths of bounded variation include, see~\cite{AmbrosioFuscoPallara2000}:
\begin{enumerate}[(i.)]
\item Left and right limits are well-defined, and one can (and we will) always take a \cadlag version. Wherever we write $(c(0),w(0))$, we implicitly mean the right limit $(c(0+),w(0+))$. 
\item Any path $(c(t),w(t))$ of bounded variation has a measure-valued derivative $\big(\dot c(\dd t),\dot w(\dd t)\big)$, and $\lVert(\dot c,\dot w)\rVert_{\TV} = \epvar(c,w)$.
\item $\BV(0,T;\RR^\sY\times\RR^\sR)$ equipped with the norm $\lVert\cdot\rVert_{L^1} + \epvar(\cdot)$ is a Banach space, and it is isometrically isomorphic to the dual of a Banach space.
\end{enumerate}

Because of the last point, the space can also be equipped with a weak-* topology, which amounts to vague convergence of both the paths $(c\super{n},w\super{n})$ and its derivatives $(\dot c,\dot w)$, defined by pairing with test functions $(\phi,\psi)\in C_0(0,T;\RR^\sY\times\RR^\sY)$. Naturally, weak-* compactness is simply characterised by norm-boundedness. Unfortunately, the weak-* topology is not metric, and hence difficult to use for stochastic analysis. Nevertheless, norm-boundedness is known to yield compactness in a slightly stronger topology~\cite[Prop.~3.13]{AmbrosioFuscoPallara2000}, which we call the \emph{hybrid} topology\footnote{The hybrid topology is usually called the weak-* topology. We name it differently to distinguish it from the functional analytically defined weak-* topology. The two topologies coincide on compact sets; in infinite dimensions the distinction becomes more subtle, see \cite{HeidaPattersonRenger2016}.}, defined through the convergence:
\begin{align*}
  (c\super{n},w\super{n})\xrightarrow{\text{hybrid}} (c,w)
    \iff
    &\lVert (c\super{n},w\super{n})-(c,w)\rVert_{L^1}\to 0     \qquad\text{and}\\
    &\big\langle (\phi,\psi),(\dot c\super{n},\dot w\super{n})\big\rangle \to \big\langle (\phi,\psi),(\dot c,\dot w)\big\rangle \quad \forall (\phi,\psi)\in C_0(0,T;\RR^\sY\times\RR^\sR).
\end{align*}
It turns out that the hybrid topology, although not metric, is `perfectly normal', which implies that the corresponding Borel $\sigma$-algebra behaves nicely, and all probabilistic tools that we will need are valid, see~\cite[Sec.~4]{HeidaPattersonRenger2016}.

\subsection{The assumptions}
\label{subsec:assumptions}

We now state the set of assumptions under which we will prove our main result. A central role is played by the sets of concentrations that are reachable via chemical reactions:
\begin{definition}[Stoichiometric simplex]\label{def:stoich simplex}
Let $c_0 \in \RR_+^\sY$
\begin{equation}
  \sS(c):=\{\tilde c=c+\Gamma w:w\in\RR_+^\sR, \tilde c\geq0\}
\end{equation}
\begin{equation}
  \sS_\epsilon(c)=\bigcup_{\tilde c \colon \lvert c - \tilde c\rvert\leq \epsilon}\sS\left(\tilde c\right)
\end{equation}
\end{definition}

For vectors in $\RR^\sY$ or $\RR^\sR$ we write $\geq$ for the partial ordering obtained by coordinate-wise inequalities. The set of assumptions on the propensities and reaction rates are the following:

\begin{assumption}[Conditions on reaction rates]\quad
\begin{enumerate}[(i)]
\item $k\super{r,V}(c)=0$ whenever $c_y< -V^{-1}\gamma\super{r}_y$ for at least one $y\in\sY$,   
  \label{it:rrate cutoff below}
\item $\sup_{c\in \sS_\epsilon\left(c(0)\right)} \sum_{r\in\sR} \abs{ \frac1V k\super{V,r}(c)-\bk\super{r}(c)}\to0$ for all $\epsilon> 0$ and $c(0)\in\RR^\sY_+$,
  \label{it:rrate convergence}
\item $\bk \in C^1(\RR^\sY_+;\RR^\sR_+)$, 
  \label{it:rrate continuous}
\item $\sup_{\tilde c\in\sS_\epsilon(c)}\lvert \bk(\tilde c)\rvert\vee \lvert \nabla_c\bk(c)\rvert<\infty$ for all $c\in\RR^\sY_+$ and $\epsilon>0$,
  \label{it:rrate upper bound}
\item $\bk(\hat c)\geq \bk(c)$ for all $\hat c\geq c$ in $\RR_+^\sY$,
  \label{it:rrate monotonicity}
\item there exists a strictly increasing bijection $\psi:\lbrack 0,1\rbrack\to [0,1]$ such that
  \begin{equation*}
    \bk\super{r}(\delta c) \geq \psi(\delta) \bk\super{r}(c)\qquad \text{for all } c\in\RR_+^\sY, \delta>0 \text{ and } r\in\sR.
  \end{equation*}
  \label{it:rrate super-homogeneity}
\end{enumerate}
\label{ass:rrate conditions}
\end{assumption}
The first assumption is needed to make sure that the stochastic model does not allow for negative concentrations. No assumptions related to boundedness or compactness of the stoichiometric simplices $\sS(c(0))$ are required; the only assumption that is needed is~\eqref{it:rrate upper bound}: that the reaction rates remain bounded on these simplices.
Furthermore, the superhomogeneity assumption~\eqref{it:rrate super-homogeneity} holds for most practical purposes, in particular for models with mass-action kinetics. We expect that the $C^1$-regularity can be relaxed to a locally Lipschitz condition, and that the monotonicity is only required in regions where the rates are small.
Taken together \eqref{it:rrate cutoff below} and \eqref{it:rrate convergence} imply that $c \geq 0$ is necessary in order to have $\bk\super{r}(c) > 0$.

The generality of the class of allowed reaction rates comes at the price of some regularity assumptions on the initial condition:

\begin{assumption}[Sufficiently regular initial LDP]\label{ass:cont initial ldp}
The initial measure $\mu\super{V}$ satisfies a large-deviation principle in $\RR^\sY_+$ with rate function $\I_0$ such that
\begin{enumerate}[(i)]
\item $\I_0$ is convex,
\item $\I_0$ is continuous,
\item $\mu\super{V}$ converges in distribution to $\delta_{\tilde c(0)}$ for some $\tilde c(0)\in\RR^\sY_+$,
\item $\mu\super{V}$ is exponentially tight (and hence $\I_0$ is good),
\label{it:initial exp tight}
\item $\I_0$ satisfies the conditions of Varadhan's Integral Lemma\cite[Th.~4.3.1]{Dembo1998} for linear functions, i.e. for all $z\in\RR^\sY$,
\begin{enumerate}
\item $\lim_{M\to\infty}\limsup_{V\to\infty} \tfrac1V\log\int_{z\cdot c(0)\geq M}\!\e^{V z\cdot c(0)}\,\mu\super{V}\big(c(0)\big)=-\infty$, or
\item $\limsup_{V\to\infty}\tfrac1V \log \int\!\e^{Va z\cdot c(0)}\,\mu\super{V}\big(c(0)\big)<\infty$ for some $a>1$.
\end{enumerate}
\item $\partial\I_0\big(c(0)\big)\neq\emptyset$ for all $c(0)\in\RR^\sY_+$.
\end{enumerate}
\end{assumption}

Although this list of assumptions is a bit technical, we point out that most assumptions mean that $C\super{V}(0)$ satisfy a `sufficiently nice' large-deviation principle. For thermodynamic properties, one is mostly interested in the large deviations where the process starts from the invariant measure \cite[Sec.~4]{Renger2017}, which often satisfies a large-deviation principle with all the needed assumptions. The continuity of $\I_0$ will be exploited (and are essential) in the approximation lemmas~\ref{lem:approx I rrate bounded below},\ref{lem:approx II smooth},\ref{lem:approx III flux bounded below} and \ref{lem:approx IV compact support}, and the last assumption is a technical requirement that is needed to prove the large-deviation lower bound for the mixture.

\subsection{Construction and convergence of the process}
\label{subsec:LLN}

We denote by $\PP\super{V}$ the path measure of the process $\big(C\super{V}(t),W\super{V}(t)\big)$ with jump dynamics as captured in the generator~\eqref{eq:generator} and initial distribution $\mu\super{V}\times\delta_0$. This is well-defined, as Assumptions~\ref{ass:rrate conditions}\eqref{it:rrate convergence} and \eqref{it:rrate upper bound} imply that the jump rates are uniformly bounded on each stoichiometric simplex $\sS(c)$, and hence \eqref{eq:generator} indeed generates a Markov process on $\BV(0,T;\RR^\sY\times\RR^\sR)$ (see \cite[Sect.~4]{HeidaPattersonRenger2016} for a discussion of the Borel $\sigma$-algebra of the hybrid topology, and related properties).

For technical reasons we shall also consider the dynamics obtained by perturbing the jump rates using exponentials of $\zeta\in C_c(0,T;\RR^\sR)$, leading to the time dependent generator
\begin{equation}
  (\sQ\super{V}_{\zeta,t}\Phi)(c,w) := \sum_{r\in\sR} k\super{V,r}(c) e^{\zeta(t)\cdot \gamma\super{r}} \big\lbrack\Phi(c+\tfrac1V\gamma\super{r},w+\tfrac1V\mathds1_{r})-\Phi(c,w)\big\rbrack.\\
\label{eq:perturbed generator}
\end{equation}
Since the jump rates remain uniformly bounded under the perturbation, this generator also defines a path measure $\PP\super{V}_\zeta$ with initial condition $\mu\super{V}\times\delta_0$.

In the interests of brevity we merely state the laws of large numbers for these measures, using the fact that the equations
\begin{equation}
  \begin{cases}
    \dot c(t)=\Gamma\dot w(t),\\
    \dot w(t)=\bk\super{r}\big(c(t)\big) e^{\zeta(t)\cdot\gamma\super{r}},\\
  \end{cases}
\label{eq:perturbed equation general}
\end{equation}
are well posed in $W^{1,1}(0,T;\RR^\sY\times\RR^\sR)$ for non-negative initial data; this may be checked by a Picard--Lindel\"of argument.
The basic ideas of the convergence proof go back to Kurtz~\cite{Kurtz1970,Kurtz1972}.

\begin{proposition} \label{th:unperturbed convergence det}
Let $\zeta\in C_c(0,T;\RR^\sR)$, Assumption~\ref{ass:rrate conditions} hold and suppose $\widetilde \mu\super{V}$ converges narrowly to $\delta_{(\tilde c(0),0)}$. Then the laws $\widetilde\PP_\zeta\super{V}$ of the Markov processes with initial conditions $\widetilde \mu\super{V}$ and dynamics given by~\eqref{eq:perturbed generator} converge narrowly to the delta measure concentrated on the $(c,w) \in W^{1,1}(0,T;\RR^\sY\times\RR^\sR)$ that is the unique solution to \eqref{eq:perturbed equation general} with initial data $(\tilde c(0), 0)$.
\end{proposition}

Note that this result includes the cases of random initial conditions $\widetilde\mu\super{V} = \mu\super{V}$ as in Assumption~\ref{ass:cont initial ldp}, as well as the case of deterministic initial conditions $\widetilde\mu\super{V}=\delta_{(\tilde c\super{V}(0),0)}$ where $\tilde c\super{V}(0)\to \tilde c(0)$.
Note also that narrow convergence of probability measures on a metric space (convergence in distribution of the associated random variables) to a deterministic limit implies convergence in probability; this can readily be generalised to the hybrid topology on the space of bounded variation paths.

\section{Analysis of the rate functional}
\label{sec:analysis of rate}
A detailed knowledge of the properties of the rate function allows for a more concise presentation of the LDP, so these properties are developed here before we embark on the stochastic aspects of the proof. It will be practical to prove a dual, variational formulation of the rate functional:
\begin{align}
  &\tilde\J(c,w) = \begin{cases}
                      \displaystyle\sup_{\zeta\in C^1_c(0,T;\RR^\sR)} G(c,w,\zeta), &\text{if } \dot c=\Gamma\dot w, \\
                      \infty, &\text{otherwise},
                    \end{cases}
\label{eq:LDP rate as sup}
\intertext{where}
  &G(c,w,\zeta):=\int_0^T\!\big\lbrack \zeta(t)\cdot \dot w(\dd t) - H\big(c(t),\zeta(t)\big) \big\rbrack\,\dd t,
\label{eq:G}\\
  &H(c,\zeta):= \sum_{r\in \sR} \bk\super{r}(c) \big(\e^{\zeta\super{r}}-1\big).
\label{eq:H}
\end{align}

\begin{remark}
$\tilde{\J} \colon \BV(0,T;\RR^\sY\times\RR^\sR_+) \rightarrow [0,\infty]$ is lower semincontinuous with respect to the hybrid topology on $\BV(0,T;\RR^\sY\times\RR^\sR_+)$ since for any $\zeta\in C_0(0,T;\RR^\sR)$ the function $(c,w) \mapsto G(c,w,\zeta)$ is hybrid continuous.
\end{remark}

\begin{remark} One can also rewrite the rate functional as a convex dual without restricting to pairs that satisfy the continuity equation: 
\begin{equation*}
  \tilde\J(c,w)=\sup_{\substack{\xi\in C_c^1(0,T;\RR^\sY) \\ \zeta\in C_c^1(0,T;\RR^\sR)}} \int_0^T\!\zeta(t)\cdot \dot w(\dd t) + \int_0^T\!\xi(t)\cdot \dot c(\dd t) - \int_0^T\!H(c(t),\zeta(t))\,\dd t.
\end{equation*}
A straight-forward calculation then shows that the rate functional reduces to \eqref{eq:LDP rate} if the continuity equation is satisfied, and $\infty$ otherwise. The variation over the dual variable to $\dot c$ corresponds in some sense to zero-probability fluctuations in the continuity equation. Therefore it is more natural to omit that supremum, which also shortens notation considerably.
\end{remark}

\subsection{Characterisation of the domain}
\label{subsec:domain}

This section is devoted to the proof that both formulations of the rate functional coincide. For the relative entropy formulation $\J$ of the rate functional, it is built into the definition~\eqref{eq:LDP rate} that $(c,w)\in W^{1,1}(0,T;\RR^\sY\times\RR^\sR_+)$ for finite $\J(c,w)$. The following Lemma says that the concentrations remain non-negative.
\begin{lemma} Let $(c,w)\in\BV(0,T;\RR^\sY\times\RR^\sR_+)$ and $c(0)\geq0$. If $\J(c,w)<\infty$ then $c\geq0$.
\label{lem:nonnegative concentrations}
\end{lemma}
\begin{proof} Assume on the contrary that one may find $t_1,y_1$ such that $c_{y_1}(t_1) < 0$.  By definition $\J(c,w) < \infty$ implies $c_{y_1}$ (has a representative that) is absolutely continuous so one may take $0 \leq t_2 < t_1$ such that $0 \geq c_{y_1}(t_2) > c_{y_1}(t_1)$.
This implies the existence of $r_1 \in\sR$ such that $\gamma_y\super{r} < 0$ and $\int_{t_2}^{t_1} \dot w\super{r_1}(s) \dd s > 0$ so \eqref{eq:LDP rate} requires $\bk\super{r_1}(c(s)) > 0$ almost everywhere in $[t_2,t_1]$.
However from Assumption~\ref{ass:rrate conditions} parts \eqref{it:rrate cutoff below} and \eqref{it:rrate convergence} one sees that $\bk\super{r_1}(c(s)) = 0$ for all $s \in [t_2,t_1]$.

\end{proof}

In order to compare $\J$ to the variational formulation~$\tilde\J$ we need to prove the same regularity result for $\tilde\J$:

\begin{lemma} Let $(c,w)\in\BV(0,T;\RR^\sY_+\times\RR^\sR_+)$. If $\tilde\J(c,w)<\infty$ then $(c,w)\in W^{1,1}(0,T;\RR^\sY\times\RR^\sR_+)$ and $c\geq0$.
\label{lem:regularity}
\end{lemma}

\begin{proof}
Let $(c,w)\in\BV(0,T;\RR^\sY_+\times\RR^\sR_+)$ and $\tilde\J(c,w)<\infty$; the proof is carried out in three stages:
\begin{enumerate}
\item $\dot w$ is a non-negative measure,
\item $\dot w(\dd t) = \dot w(t)\,\dd t$ for some density $\dot w\in L^1(0,T;\RR^\sR_+)$, 
\item $\dot c(t) = \Gamma \dot w(t)$, 
\item $c \geq 0$.
\end{enumerate}
For the first point note that the existence of $\dot w$ as a (signed) vector measure of finite total variation follows from \cite[Thrm. 2.13]{HeidaPattersonRenger2016}.  Suppose now that there is some $r\in\sR$ and a measurable set $A \subset (0,T)$ such that $\dot w\super{r}(A) < 0$.
Using the Hahn decomposition and the regularity of Borel measures on the metric space $(0,T)$ (\cite[Thrm. 7.1.7]{Bogachev2007} or \cite[Lem. 1.34]{Kallenberg2002}) one has the existence of a closed $B \subset A$ with $\dot w\super{r}(B) < 0$.  Define $\zeta_n \in C_c^1(0,T;\RR_-^\sR)$ by
\begin{equation*}
\zeta_n\super{r^\prime}(t) =
\begin{cases}
0 & r^\prime \neq r \\
-n \varphi(t) & r^\prime = r
\end{cases}
\end{equation*}
for some $\varphi \in C^1_c(0,T; [0,1])$ such that $\II_B \leq \varphi \leq \II_A$.
On can now check that $\lim_n G(c,w,\zeta_n) = + \infty$, which contradicts $\tilde\J < \infty$ so there cannot be any $r$ for which $\dot w\super{r}$ takes negative values.

For the absolute continuity suppose that there is an $r\in\sR$ and a measurable set $A\subset (0,T)$ such that $\dot w\super{r}(A) = \delta > 0$, but $\abs{A} = 0$, where we write $\abs{\cdot}$ for Lebesgue measure.
By the regularity result already mentioned in this proof we have the existence of closed sets $F_n$ and open sets $G_n$ such that $F_n \subset A \subset G_n$ with $\dot w\super{r}(G_n \setminus F_n) < \frac1n$ and $\abs{G_n} \leq \frac1n$.
Define $\zeta_n \in C_c^1(0,T;\RR_+^\sR)$ by
\begin{equation*}
\zeta_n\super{r^\prime}(t) =
\begin{cases}
0 & r^\prime \neq r \\
-\log\abs{G_n} \varphi_n(t) & r^\prime = r
\end{cases}
\end{equation*}
for some $\varphi \in C^1_c(0,T; [0,1])$ such that $\II_{F_n} \leq \varphi \leq \II_{G_n}$ to get a contradiction as in the proof that $\dot w \geq 0$.
The Radon-Nikodym theorem thus allows us with a little abuse of notation to write $\dot w\super{r}(\dd t) = \dot w\super{r}(t)\dd t$ for $\dot w\super{r} \in L^1(0,T;\RR∏)$.

The proof that $c(t)\geq0$ is the same as in Lemma~\ref{lem:nonnegative concentrations}, where now we have on the non-null set $B\subset(0,T)$,
\begin{equation*}
  \J(c,w)\geq \sup_{\zeta\super{r}\in C_c^1(B)} \int_B\! \zeta\super{r}(t)\cdot\dot w\super{r}(t) - 0 = \infty.
\end{equation*}
\end{proof}

\begin{proposition} $\J=\tilde\J$.
\label{prop:J=tilde J}
\end{proposition}

\begin{proof}
Let $(c,w)\in\BV(0,T;\RR^\sY_+\times\RR^\sR_+)$ (possibly with $\J(c,w)=\infty$). If $(c,w)\notin W^{1,1}(0,T;\RR^\sY_+\times\RR^\sR_+)$ then by Lemma~\ref{lem:regularity} both $\tilde\J(c,w)=\infty=\J(c,w)$. Now assume that $(c,w)\in W^{1,1}(0,T;\RR^\sY_+\times\RR^\sR_+)$. We can then write $G(c,w,\zeta)=\sum_{r\in\sR}\int_0^T\!g\super{r}\big(c(t),\dot w\super{r}(t),\zeta\super{r}(t)\big)\,\dd t$ where
\begin{equation*}
  g\super{r}(c,j\super{r},\zeta\super{r}) := \zeta\super{r} j\super{r} - \bk\super{r}(c)\big(\e^{\zeta\super{r}}-1\big).
\end{equation*}
We now show that
\begin{align}
  \J(c,w)&=\sup_{\zeta:(0,T)\to\RR^\sR} G(c,w,\zeta)
   = \sup_{\zeta\in L^\infty(0,T;\RR^\sR)} G(c,w,\zeta) \notag\\
  &= \sup_{\zeta\in C^1_b(0,T;\RR^\sR)} G(c,w,\zeta)
   = \sup_{\zeta\in C^1_c(0,T;\RR^\sR)} G(c,w,\zeta)
   = \tilde\J(c,w).
\label{eq:sup over different sets}
\end{align}

The first equality in~\eqref{eq:sup over different sets} can be calculated directly through the pointwise supremum. For the second equality, we construct, for each $t\in(0,T)$ and $r\in\sR$, an explicit (pointwise) maximising sequence $\zeta\super{r}_n(t)$ for $\sup_{\zeta\super{r}} g\super{r}\big(c(t),\dot w\super{r}(t),\zeta\super{r}(t)\big)$ as, see Figure~\ref{fig:gr},
\begin{equation*}
  \zeta\super{r}_n(t):=
    \begin{cases}
      \log\frac{w\super{r}(t)}{\bk\super{r}(c(t))} \wedge n,    &\bk\super{r}(c(t))>0 \text{ and } \dot w\super{r}(t)>0,\\
      -n,                                                       &\bk\super{r}(c(t))>0 \text{ and } \dot w\super{r}(t)=0,\\
      n,                                                        &\bk\super{r}(c(t))=0 \text{ and } \dot w\super{r}(t)>0,\\
      0,                                                        &\bk\super{r}(c(t))=0 \text{ and } \dot w\super{r}(t)=0.
    \end{cases}
\end{equation*}
\begin{figure}[h]
  \centering
  \subfloat[$\bk\super{r}(c)>0,j\super{r}>0$]{%
    \quad
    \begin{tikzpicture}[scale=0.7]
      \tikzstyle{every node}=[font=\footnotesize]
      \draw[->](-2,0)--(2,0);
      \draw(0,-1)--(0,1);
      \draw(-1.2,-1) ..controls(0.43,0.4) and (1.4,1.35)..(1.9,-1);
      \draw[dotted](0,0.45)--(1,0.45); 
    \end{tikzpicture}
    \,
  }
  \subfloat[$\bk\super{r}(c)>0,j\super{r}=0$]{%
    \quad
    \begin{tikzpicture}[scale=0.7]
      \tikzstyle{every node}=[font=\footnotesize]
      \draw[->](-2,0)--(2,0);
      \draw(0,-1)--(0,1);
      \draw(-2,0.42) ..controls(0,0.45) and (0.4,-0.3)..(0.9,-1);
      \draw[dotted](0,0.45)--(-2,0.45); 
    \end{tikzpicture}
    \,
  }
  \subfloat[$\bk\super{r}(c)=0,j\super{r}>0$]{%
    \quad
    \begin{tikzpicture}[scale=0.7]
      \tikzstyle{every node}=[font=\footnotesize]
      \draw[->](-2,0)--(2,0);
      \draw(0,-1)--(0,1);
      \draw(-2,-1)--(2,1);
    \end{tikzpicture}
    \,
  }
  \subfloat[$\bk\super{r}(c)=0,j\super{r}=0$]{%
    \quad
    \begin{tikzpicture}[scale=0.7]
      \tikzstyle{every node}=[font=\footnotesize]
      \draw[->](-2,0)--(2,0);
      \draw[thick](-2,0)--(2,0);
      \draw(0,-1)--(0,1);
    \end{tikzpicture}
    \quad
  }
  \caption{The function $\zeta\super{r}\mapsto g\super{r}(c,j\super{r},\zeta\super{r})$.}
  \label{fig:gr}
\end{figure}

\noindent Then each $\zeta_n\in L^\infty(0,T;\RR^\sR)$ and $g\super{r}\big(c(t),\dot w\super{r}(t),\hat\zeta\super{r}(t)\big)$ is non-decreasing in $n$ and non-negative. Moreover, $g\super{r}\big(c(t),\dot w\super{r}(t),\hat\zeta\super{r}(t)\big)$ converges pointwise in $t\in(0,T)$ and $r\in\sR$ as $n\to\infty$ to the pointwise supremum. Hence by monotone convergence~
\begin{equation*}
  \lim_{n\to\infty} \sum_{r\in\sR} \int_0^T\! g\super{r}\big(c(t),\dot w\super{r}(t),\hat\zeta\super{r}_n(t)\big) = \sum_{r\in\sR} \int_0^T\! \sup_{\zeta\super{r}} g\super{r}\big(c(t),\dot w\super{r}(t),\zeta\super{r}(t)\big)\,\dd t.
\end{equation*}
This shows that the pointwise supremum on the left of \eqref{eq:sup over different sets} can be taken over $L^\infty(0,T;\RR^\sR)$.

For the third equality in \eqref{eq:sup over different sets} it suffices to show that for any $\zeta\in L^\infty(0,T;\RR^\sR)$ the integrand can be approximated by a sequence in $C_b^2(0,T;\RR^\sR)$. For an arbitrary $\zeta\in L^\infty(0,T;\RR^\sR)$ consider the convolutions with smoothing kernels $\theta_\delta$ for $\delta>0$ that weakly converges to the Dirac measure at 0 as $\delta \rightarrow 0$.
In the convolutions we extended the function $\zeta$ to zero outside the integral $(0,T)$. Since $\zeta\in L^1(\RR;\RR^\sR)$ this sequence $\zeta \ast \theta_\delta$ converges strongly in $L^1(\RR;\RR^\sR)$ to $\zeta$ as $\delta \rightarrow 0$, see~\cite[App.~C.4]{Evans2002}. By a partial converse of the Dominated Convergence Theorem~\cite[Th.~IV.9]{Brezis1983}, after passing to a subsequence $\zeta_n(t)\super{r}:=(\zeta*\theta_{\delta_n}\super{r})(t)$ converges pointwise $t$-almost everywhere.
Then the exponential $-\bk\big(c(t)\big)\big(\e^{\zeta\super{r}_n(t)}-1\big)$ integrand part of $G(c,w,\zeta_n)$ also converges pointwise for almost every $t$. Moreover, we can bound
\begin{equation*}
  \lVert\bk\rVert_\infty \geq -\bk\big(c(t)\big)\big(\e^{\zeta\super{r}_n(t)}-1\big) \geq -\lVert\bk\rVert_\infty \e^{\lVert\zeta_n\rVert_{L^\infty(0,T;\RR^\sR)}} \geq -\lVert\bk\rVert_\infty \e^{\lVert\zeta\rVert_{L^\infty(0,T;\RR^\sR)}},
\end{equation*}
and hence by dominated convergence
\begin{equation*}
  -\sum_{r\in\sR}\int_0^T\!\bk\big(c(t)\big)\big(\e^{\zeta\super{r}_n(t)}-1\big)\,\dd t \to -\sum_{r\in\sR}\int_0^T\!\bk\big(c(t)\big)\big(\e^{\zeta\super{r}(t)}-1\big)\,\dd t.
\end{equation*}
Clearly the linear part $\sum_{r\in\sR}\int_0^T\!\zeta\super{r}_n(t)\dot w\super{r}(t)\,\dd t$ of $G(c,w,\zeta_n)$ converges to $\sum_{r\in\sR}\int_0^T\!\zeta\super{r}(t)\dot w\super{r}(t)\,\dd t$, and so $G(c,w,\zeta_n)\to G(c,w,\zeta)$. This proves the third equality in \eqref{eq:sup over different sets}.

For the fourth equality, take any $\zeta\in C_b^1(0,T;\RR^\sR)$, and approximate with $\zeta\eta_\delta\in C_c^1(0,T;\RR^\sR)$ where
\begin{equation}
  \eta_\delta(t):=\begin{cases}
    0,    & t\in(0,\delta\rbrack\cup\lbrack T-\delta,T),\\
    1,    & t\in\lbrack2\delta,T-2\delta\rbrack,\\
    \text{smooth between } 0 \text{ and } 1, & t\in\lbrack\delta,2\delta\rbrack\cup\lbrack T-2\delta,T-\delta\rbrack.
  \end{cases}
\label{eq:horizontal cutoff}
\end{equation}
Then, as $\delta\to0$,
\begin{equation*}
  G(c,w,\zeta\eta_\delta) = \int_0^T\!\dot w(t)\zeta(t)\eta_\delta(t)\,\dd t - \sum_{r\in\sR}\int_0^T\!\bk\super{r}\big(c(t)\big)(\e^{\zeta_\delta(t)}-1)\,\dd t \to G(c,w,\zeta\eta_\delta)\to G(c,w,\zeta),
\end{equation*}
where for the linear part we use that $\zeta\eta_\delta\to\zeta$ weakly-* in $L^\infty(0,T;\RR^\sR)$, and for the nonlinear part we use dominated convergence.
\end{proof}

\subsection{Approximation by regular curves}
\label{subsec:approximations}

A common challenge in proving a large-deviations lower bound for a Markov process is to approximate any curve of finite rate by curves for which one can perform a change-of-measure. In the setting of our paper, this set of sufficiently regular curves will be defined as:
\begin{multline}
  \sA := \Big\{(c,w) \in \BV\big(0,T;\RR^\sY_+\times\RR^\sR_+\big)\cap \AC\big(0,T;\RR^\sY\times\RR^\sR\big) \colon \\
     \zeta:=\log\mfrac{\dot w}{\bk(c)}\in C_c^1\big(0,T;\RR^\sR\big), \quad
    \dot c=\Gamma\dot w, \quad c,w,\dot w\geq0, \quad w(0)=0  \Big\}.
\label{eq:suff regular curves for Girsanov}
\end{multline}
Observe that this set requires compactly supported perturbations, whereas the change-of-measure Theorem~\ref{th:change of measure general} only requires boundedness. However, the compact support will be needed to control the end point in the tilting arguments, Lemmas~\ref{lem:lower bound for suff reg curves} and \ref{lem:upper bound dynamic}.

This section is dedicated to the proof of the required approximation result using a sequence of four approximation lemmas.
We repeatedly exploit the lower semi-continuity of $\J$ to show that if $\lim_{\delta \searrow 0}(c_\delta,w_\delta) = (c,w)$ in the hybrid topology, then $\liminf_{\delta \searrow 0} \J(c_\delta,w_\delta) \geq \J(c,w)$.

\begin{lemma}[Approximation I] Let $\mu\super{V}$ satisfy Assumption~\ref{ass:cont initial ldp} and $\bk$ satisfy Assumptions~\ref{ass:rrate conditions}\eqref{it:rrate continuous},\eqref{it:rrate upper bound},\eqref{it:rrate monotonicity} and \eqref{it:rrate super-homogeneity}. Given $(c,w)\in \BV\big(0,T;\RR^\sY_+\times\RR^\sR_+\big)$ such that $\J(c,w)<\infty$, there exists a sequence $(c_\delta,w_\delta)_\delta\subset\BV\big(0,T;\RR^\sY_+\times\RR^\sR_+\big)$ such that:
\begin{enumerate}[(i)]
\item $c_\delta(0)\to c(0)$ and $(c_\delta,w_\delta) \hybridto (c,w)$ as $\delta\to0$, \label{it:approx I conv} 
\item $\I_0(c_\delta(0))+\J(c_\delta,w_\delta)\to\I_0(c(0))+\J(c,w)$ as $\delta\to0$, \label{it:approx I rate conv}
\item $\inf_{t\in(0,T),r\in\sR} \bk\super{r}(c_\delta(t))>0$ for any $\delta>0$, \label{it:approx I bound}
\end{enumerate}
\label{lem:approx I rrate bounded below}
\end{lemma}

\begin{proof} Without loss of generality we may assume that for each reaction $r$ there exists a concentration $\hat c\super{r}\in\RR_y^+$ for which $\bk\super{r}(\hat c\super{r})>0$. Set $\hat c =\sum_{r\in\sR} \hat c\super{r}$, so that by the assumed monotonicity,
\begin{equation*}
  \min_{r\in\sR} \bk\super{r}(\hat c) >0.
\end{equation*}
For $\delta>0$ define
\begin{align*}
  c_\delta(t):=\delta \hat c + (1-\delta)c(0) +  \Gamma w_\delta(t),
    &&\text{and}&&
  w_\delta(t):=(1-\delta) w,
\end{align*}
so that $c_\delta(t)=\delta \hat c + (1-\delta)c(t)\geq0$.

The limits~\eqref{it:approx I conv} are trivial. The lower bound~\eqref{it:approx I bound} follows by the monotonicity and superhomogeneity Assumptions~\ref{ass:rrate conditions}\eqref{it:rrate monotonicity} and \eqref{it:rrate super-homogeneity}:
\begin{equation}
  \inf_{t\in(0,T),r\in\sR} \bk\super{r}\big(c_\delta(t)\big) \geq \min_{r\in\sR} \bk\super{r} (\delta\hat c) \geq \psi(\delta) \min_{r\in\sR} \bk\super{r}(\hat c) >0.
\label{eq:approx I rrate explicit lower bound}
\end{equation}

For the limits~\eqref{it:approx I rate conv}, the convergence of $\I_0(c_\delta(0))$ follows by Assumption~\ref{ass:cont initial ldp}.
Since $\liminf_{\delta \searrow 0} \J(c_\delta,w_\delta) \geq \J(c,w)$ it remains to check $\limsup_{\delta \searrow 0} \J(c_\delta,w_\delta) \leq \J(c,w)$. Using the fact that $\bk\super{r}(c_\delta(t)) \geq \psi(1-\delta)\bk\super{r}(c(t))$ by the same argument as~\eqref{eq:approx I rrate explicit lower bound} above, we can rewrite and estimate:
\begin{align}
  s\!\left(\dot{w}\super{r}_\delta(t) \mid  \bk\super{r}\left(c_\delta(t)\right) \right)
    &= s\!\left(\dot{w}\super{r}_\delta(t) \mid  \bk\super{r}\left(c(t)\right) \right)
  + \dot{w}\super{r}_\delta(t) \log\left(\frac{\bk\super{r}\left(c(t)\right)}{\bk\super{r}\left(c_\delta(t)\right)} \right)
  -\bk\super{r}\left(c(t)\right)+\bk\super{r}\left(c_\delta(t)\right)
  \label{eq:approx I splitting kdelta}\\
  &\leq
(1-\delta) s\!\left(\dot{w}\super{r}(t) \mid  \bk\super{r}\left(c(t)\right) \right)
  +\delta   \bk\super{r}\left(c(t)\right) \notag\\
  &\qquad +  (1-\delta) \dot{w}\super{r}(t) \log \mfrac{1-\delta}{\psi(1-\delta)}
  +\abs{\bk\super{r}\left(c(t)\right)-\bk\super{r}\left(c_\delta(t)\right)}.
  \label{eq:approx I splitting estimate}
\end{align}
Summing over $r$ and integrating over $t$ shows that, for $\delta$ sufficiently small,
\begin{equation*}
\J(c_\delta,w_\delta)
\leq
(1-\delta) \J(c,w) + \delta T \sup_{\tilde{c}\in\sS_{2\delta}(c(0))}\lvert \bk(\tilde c)\rvert
  +  (1-\delta) \log\mfrac{1-\delta}{\psi(1-\delta)} \norm{\dot w}_{L^1}
  + \sqrt{\lvert\sR\rvert}\Lip(\bk) \norm{c_\delta - c}_{L^1}.
\end{equation*}
Using Assumption~\ref{ass:rrate conditions}\eqref{it:rrate upper bound} it follows that all but the first term on the right-hand side vanish as $\delta \searrow 0$ and the result is established.
\end{proof}

For smoothing purposes we make use of convolutions with the heat kernels $\theta_\epsilon \colon \RR \rightarrow \RR_+; t \mapsto \exp(-t^2/2\epsilon) / \sqrt{2\pi \epsilon}$.

\begin{lemma}[Approximation II] Let $\mu\super{V}$ satisfy Assumption~\ref{ass:cont initial ldp} and $\bk$ satisfy Assumptions~\ref{ass:rrate conditions}\eqref{it:rrate continuous} and~\eqref{it:rrate upper bound}. Given $(c,w)\in \BV\big(0,T;\RR^\sY_+\times\RR^\sR_+\big)$ such that $\J(c,w)<\infty$ and $\inf_{t\in(0,T),r\in\sR} \bk\super{r}(c(t))>0$, there exists a sequence $(c_\delta,w_\delta)_\delta\subset C^\infty_b(0,T;\RR^\sY_+\times\RR^\sR_+)$ such that:
\begin{enumerate}[(i)]
\item $c_\delta(0)\to c(0)$ and $(c_\delta,w_\delta) \hybridto (c,w)$ as $\delta\to0$, \label{it:approx II conv} 
\item $\I_0(c_\delta(0))+\J(c_\delta,w_\delta)\to\I_0(c(0))+\J(c,w)$ as $\delta\to0$, \label{it:approx II rate conv}
\item $\inf_{t\in(0,T),r\in\sR} \bk\super{r}(c_\delta(t))>0$ for any sufficiently small $\delta>0$. \label{it:approx II bound}
\end{enumerate}
\label{lem:approx II smooth}
\end{lemma}

\begin{proof}
Define\begin{align*}
  c_\delta(t):=c(0) + (w*\theta_\delta)(0) + \Gamma w_\delta(t)
    &&\text{and}&&
  w_\delta(t):=(w*\theta_\delta)(t) - (w*\theta_\delta)(0),
\end{align*}
where in the convolutions we extend $w$ constantly to $w(0)$ and $w(T)$ outside the interval $(0,T)$. Observe that the definition is sound in the sense that $c_\delta(t)=(c*\theta_\delta)(t)\geq0$ and $w_\delta,\dot w_\delta\geq0$.
Since $(c,w)\in W^{1,1}(0,T;\RR^\sY\times\RR^\sR)$ by Lemma~\ref{lem:regularity}, the desired convergence~\eqref{it:approx II conv} of the sequence can be shown by adapting the results in \cite[App.~C.4]{Evans2002} to mollifiers with non-compact support. Similarly to the proof of Proposition~\ref{prop:J=tilde J}, we pass to a (relabelled) subsequence such that in fact $\dot w_\delta(t)\to\dot w(t)$ pointwise in almost every $t\in (0,T)$.

To show the lower bound~\eqref{it:approx II bound}, observe that since $c$ is continuous (on the compact interval $\lbrack0,T\rbrack$) by Lemma~\ref{lem:regularity}, $c_\delta\to c$ uniformly, see \cite[App.~C.4]{Evans2002}. Moreover, by the continuity of $\bk$, we know there exists a $\tau>0$ such that for $\delta<\tilde\delta$ and any $\tilde c \in \RR^\sY_+$ with $\lvert \tilde c - c\rvert<\tau$ there holds $\lvert \bk(\tilde c)-\bk(c)\rvert<\tfrac12\inf_{t\in(0,T),r\in\sR}\bk\super{r}(c(t))$. From the uniform convergence of $c_\delta$ we get the existence of a $\tilde\delta>0$ such that for any $\delta<\tilde\delta$ and any $t\in\lbrack0,T\rbrack$, there holds $\lvert c(t)-c_\delta(t)\rvert<\tau$. Therefore, for any $\delta<\tilde\delta$  and so $\lvert\bk(c_\delta(t))-\bk(c(t))\rvert<\tfrac12\inf_{t\in(0,T),r\in\sR}\bk\super{r}(c(t))$, from which we deduce the lower bound~\eqref{it:approx II bound}:
\begin{equation}\label{eq:smoothed rate lbound}
  \bk(c_\delta(t))\geq \bk(c(t))-\tfrac12\inf_{t\in(0,T),r\in\sR}\bk\super{r}(c(\tilde t)) \geq \tfrac12\inf_{\tilde t\in(0,T),r\in\sR}\bk\super{r}(c(\tilde t)).
\end{equation}

The convergence $\I_0\big(c_\delta(0)\big)=\I_0\big(c(0)+(w*\theta_\delta)(0)\big)\to\I_0\big(c(0)\big)$ follows by Assumption~\ref{ass:cont initial ldp}. For the convergence of $\J(c_\delta,w_\delta)$, we can bound the integrand, similarly as in \eqref{eq:approx I splitting kdelta},
\begin{equation}
  0  \leq s\!\left(\dot{w}\super{r}_\delta(t) \mid  \bk\super{r}\left(c_\delta(t)\right) \right) 
     \leq s\!\left(\dot{w}\super{r}_\delta(t) \mid  \bk\super{r}\left(c(t)\right) \right)
  + a\dot{w}\super{r}_\delta(t) 
  + \left\lvert \bk\super{r}\left(c(t)\right)-\bk\super{r}\left(c_\delta(t)\right) \right\rvert,
\label{eq:s upper bound dom conv 2}
\end{equation}
where
\begin{equation*}
  \log\Big(\mfrac{\bk\super{r}\left(c(t)\right)}{\bk\super{r}\left(c_\delta(t)\right)}\Big) \stackrel{\eqref{eq:smoothed rate lbound}}{\leq} \log\Big(\mfrac{2\sup_{t\in (0,T)}\bk(c(t))}{\inf_{t\in(0,T)}\bk(c(t))}\Big) =: a \in\lbrack0,\infty).
\end{equation*}
By the assumed continuity of the reaction rates $s\big(\dot w_\delta\super{r}(t) \vert k\super{r}(c_\delta(t))\big)\to s\big(\dot w\super{r}(t) \vert k\super{r}(c(t))\big)$ pointwise in $t\in(0,T)$. If we can prove that, after summing over $\sR$ and integrating over $(0,T)$, the right-hand side in \eqref{eq:s upper bound dom conv 2} converges to a finite integral, then $\J(c_\delta,w_\delta)\to\J(c,w)$ by a generalisation of the Dominated Convergence Theorem, see~\cite[Th.~1.8 \& following remark]{Liebloss2001}.

Naturally the last two terms converge:
\begin{equation*}
  \sum_{r\in\sR}\int_0^T\!\Big\lbrack a\dot{w}\super{r}_\delta(t) 
  + \left\lvert \bk\super{r}\left(c(t)\right)-\bk\super{r}\left(c_\delta(t)\right) \right\rvert \big\rbrack \,\dd t \to 
  a \lVert\dot{w}\rVert_{L^1}<\infty.
\end{equation*}
The convergence of the entropic part can be proven analogue to~\cite[Lem.~4.11]{Renger2017}. By lower semicontinuity,
\begin{equation*}
  \liminf_{\delta\to0} \sum_{r\in\sR}\int_0^T\!s\big(\dot w\super{r}_\delta(t) \vert \bk\super{r}(c(t))\big)\,\dd t
    \geq
  \sum_{r\in\sR}\int_0^T\! s\big(\dot w\super{r}(t) \vert \bk\super{r}(c(t))\big)\,\dd t.
\end{equation*}
On the other hand, by Jensen's inequality,
\begin{align*}
  &\sum_{r\in\sR}\int_0^T\! s\big(\dot w\super{r}_\delta(t) \vert \bk\super{r}(c(t))\big)\,\dd t 
  \leq \sum_{r\in\sR}\int_0^T\!\Big(s\big(\dot w\super{r}(\cdot)  \vert \bk\super{r}(c(t))\big)*\theta_\epsilon\Big)(t) \,\dd t\\
  &\qquad = \sum_{r\in\sR}\int_0^T\! \underbrace{(\dot w\super{r}\log\dot w\super{r}}_{\in L^1})*\theta_\epsilon(t) 
    - (\dot w\super{r}*\theta_\epsilon)(t) \underbrace{\big( 1+\log \bk\super{r}(c(t)) \big)}_{\in L^\infty}
    + \bk\super{r}\big(c(t)\big) \\
  &\qquad \to\sum_{r\in\sR}\int_0^T\! s\big(\dot w\super{r}(t) \vert \bk\super{r}(c(t))\big)\,\dd t,
\end{align*}
again by \cite[App.~C.4]{Evans2002}. Therefore the summed and integrated right-hand side of \eqref{eq:s upper bound dom conv 2} indeed converges to a finite integral, which concludes the proof of claim~\eqref{it:approx II rate conv}.

\end{proof}

\begin{lemma}[Approximation III] Let $\mu\super{V}$ satisfy Assumption~\ref{ass:cont initial ldp} and $\bk$ satisfy Assumptions~\ref{ass:rrate conditions}\eqref{it:rrate continuous},\eqref{it:rrate upper bound}, \eqref{it:rrate monotonicity} and \eqref{it:rrate super-homogeneity}.
Given $(c,w)\in C_b^\infty\big(0,T;\RR^\sY_+\times\RR^\sR_+\big)$ such that $\J(c,w)<\infty$ and $\inf_{t\in(0,T),r\in\sR} \bk\super{r}(c(t))>0$, there exists a sequence $(c_\delta,w_\delta)_\delta\subset C_b^\infty\big(0,T;\RR^\sY_+\times\RR^\sR_+\big)$ such that:
\begin{enumerate}[(i)]
\item $c_\delta(0)\to c(0)$ and $(c_\delta,w_\delta) \hybridto (c,w)$ as $\delta\to0$, \label{it:approx III conv} 
\item $\I_0(c_\delta(0))+\J(c_\delta,w_\delta)\to\I_0(c(0))+\J(c,w)$ as $\delta\to0$, \label{it:approx III rate conv}
\item $\inf_{t\in(0,T),r\in\sR} \bk\super{r}(c_\delta(t))>0$ for any $\delta>0$, \label{it:approx III bound k}
\item $\inf_{t\in(0,T),r\in\sR} \dot w_\delta\super{r}(t)>0$ for any $\delta>0$, \label{it:approx III bound w}
\item $\zeta_\delta:=\log \mfrac{\dot w_\delta}{\bk(c_\delta)} \in C_b^1(0,T;\RR^\sR)$. \label{it:approx III zeta Cb2}
\end{enumerate}
\label{lem:approx III flux bounded below}
\end{lemma}

\begin{proof}
Let $\beta\super{r},\alpha\super{r}\in\RR_+^\sY$ be the positive and negative parts of $\gamma\super{r}$, i.e. $\gamma\super{r}=\beta\super{r}-\alpha\super{r}$. For $0<\delta<1$ define
\begin{align*}
  c_\delta(t):=(1-\delta) c(0) + \delta T\sum_{r\in\sR}\alpha\super{r} + \Gamma w_\delta(t)
    &&\text{and}&&
  w_\delta(t):=(1-\delta)w(t) + \delta t,
\end{align*}
so that $c_\delta(t)=(1-\delta) c(t) + \delta\sum_{r\in\sR} \big\lbrack (T-t)\alpha\super{r} + t\beta\super{r}\big\rbrack \geq0$ and $\dot w_\delta(t)=(1-\delta) \dot w(t)+\delta\geq \delta>0$. 
Hence the sequence is admissable, and property \eqref{it:approx III bound w} holds by construction. Again, the hybrid convergence~\eqref{it:approx III conv} is trivial, and the monotonicity and superhomogeneity, Assumptions~\ref{ass:rrate conditions}\eqref{it:rrate monotonicity}, \eqref{it:rrate super-homogeneity} imply the same estimate as \eqref{eq:approx I rrate explicit lower bound}, which shows that the bound~\eqref{it:approx III bound k} is indeed retained.

The convergence $\I_0\big(c_\delta(0)\big)\to\I_0\big(c(0)\big)$ follows from the continuity of $\I_0$ and \eqref{it:approx III conv}.  As in the previous lemmas it is sufficient to show $\limsup_{\delta \searrow 0} \J(c_\delta,w_\delta) \leq \J(c,w)$ in order to establish~\eqref{it:approx I rate conv}. We can again derive estimate~\eqref{eq:approx I splitting estimate}, where the terms $\dot w\super{r}_\delta(t)\log\bk\super{r}(c(t))/\bk\super{r}(c_\delta(t))$ and $\bk\super{r}(c(t))-\bk\super{r}(c_\delta(t))$ can be dealt with in exactly the same manner as in the proof of Lemma~\ref{lem:approx I rrate bounded below}. It thus remains to show convergence of the integral $\sum_{r\in\sR} \int_0^Ts\big(\dot w\super{r}_\delta(t) \vert \bk\super{r}(c(t))\big)\,\dd t$. By the convexity of $s$ in its first argument, we get for $0<\delta<1$,
\begin{align*}
  s\big(\dot w\super{r}_\delta(t) \vert \bk\super{r}(c(t))\big) 
  &\leq
  (1-\delta) s\big(\dot w\super{r}(t) \vert \bk\super{r}(c(t))\big) + \delta s\big(1\vert \bk\super{r}(c(t))\big) \\
  &\leq
s\big(\dot w\super{r}(t) \vert \bk\super{r}(c(t))\big) - \delta\log\bk\super{r}(c(t)) + \delta \bk\super{r}(c(t)).
\end{align*}
Since the last two terms are bounded from below and above it follows that $\limsup_{\delta \searrow 0} \J(c_\delta,w_\delta) \leq \J(c,w)$.

Finally we can prove~\eqref{it:approx III zeta Cb2} for any $\delta>0$. Since the curve $(c_\delta,w_\delta)$ is smooth we only need to prove boundedness of the functions
\begin{align*}
  \zeta\super{r}_\delta(t)=\log\mfrac{\dot w\super{r}_\delta(t)}{\bk(c_\delta(t))},
  &&\text{and}&&
  \dot\zeta\super{r}_\delta(t)= \frac{\ddot w\super{r}_\delta(t)}{\dot w\super{r}_\delta(t)} - \frac{\nabla_c\bk(c_\delta(t))\cdot\dot c_\delta(t)}{\bk(c_\delta(t))}.
\end{align*}
This follows from the boundedness away from zero of $\dot w_\delta$ and $\bk(c_\delta)$, together Assumption~\ref{ass:rrate conditions}\eqref{it:rrate upper bound}.
\end{proof}

\begin{lemma}[Approximation IV] Let $\bk$ satisfy Assumptions~\eqref{it:rrate continuous},\eqref{it:rrate upper bound}. 
Given $(c,w)\in C_b^\infty\big(0,T;\RR^\sY_+\times\RR^\sR_+\big)$ such that $\J(c,w)<\infty$ and $\zeta=\log\dot w/\bk(c)\in C_b^1(0,T;\RR^\sR)$, there exists a sequence $(c_\delta,w_\delta)\in C_c^1\big(0,T;\RR^\sY_+\times\RR^\sR_+\big)$ such that:
\begin{enumerate}[(i)]
\item $c_\delta(0)\equiv c(0)$ and $(c_\delta,w_\delta) \hybridto (c,w)$ as $\delta\to0$, \label{it:approx IV conv} 
\item $\I_0(c_\delta(0))+\J(c_\delta,w_\delta)\to\I_0(c(0))+\J(c,w)$ as $\delta\to0$. \label{it:approx IV rate conv}
\end{enumerate}
\label{lem:approx IV compact support}
\end{lemma}

\begin{proof}
Given $(c,w)$ with $\zeta=\log\dot w/\bk\in C_b^1(0,T;\RR^\sR)$, we approximate $\zeta_\delta:=\zeta\eta_\delta$ where $\eta_\delta$ is the usual compactly supported function~\eqref{eq:horizontal cutoff}. Clearly $(c,w)$ satisfies the perturbed equation below, and we define, for each $\delta>0$ the path $(c_\delta,w_\delta)$ as the solution of the second perturbed equation:
\begin{align*}
  \begin{cases}
    \dot c(t)=\Gamma\dot w(t),\\
    \dot w(t)=\bk\super{r}\big(c(t)\big)\e^{\zeta\super{r}(t)},
  \end{cases} 
&&
  \begin{cases}
    \dot c_\delta(t)=\Gamma\dot w_\delta(t),\\
    \dot w_\delta(t)=\bk\super{r}\big(c_\delta(t)\big)\e^{\zeta\super{r}_\delta(t)},
  \end{cases} 
\end{align*}
both under the same initial conditions $(c(0),0)$.

Let us now introduce the matrix norm,
\begin{equation*}
  \norm{\Gamma} := \max_{r\in\sR} \abs{\gamma\super{r}}.
\end{equation*}
To prove convergence \eqref{it:approx IV conv} we first estimate for any $0\leq t\leq T$,
\begin{multline*}
\lvert\dot w_\delta(t)-\dot w(t)\rvert \\
\leq
  \sum_{r\in\sR}\left\lvert \bk\super{r}\big(c_\delta(t)\big)\e^{\zeta\super{r}_\delta(t)}
                                      -\bk\super{r}\big(c(t)\big)\e^{\zeta\super{r}_\delta(t)}\right\rvert
   +
   \sum_{r\in\sR} \left\lvert \bk\super{r}\big(c(t)\big)\e^{\zeta\super{r}_\delta(t)}
                                        -\bk\super{r}\big(c(t)\big)\e^{\zeta\super{r}(t)}\right\rvert\\
\leq
  \Lip(\bk)\lVert\Gamma\rVert\e^{\lVert\zeta\rVert_{L^\infty}}
               \int_0^t\!\lvert\dot w_\delta(\hat t)-\dot w(\hat t)\rvert\,d\hat t
       + \left( \sup_{\hat{c} \in \sS(c(0))}\sum_{r\in\sR}\bk\super{r}(\hat{c})\right)
         \max_{r\in\sR}\abs{\e^{\zeta\super{r}_\delta(t)}-\e^{\zeta\super{r}(t)}}.
\end{multline*}
From \eqref{eq:horizontal cutoff} one sees that $\zeta_\delta(t) = \zeta(t)$ except on two intervals each with length no more than $2\delta$. Gronwall's inequality yields 
\begin{equation}
  \lvert\dot w_\delta(t)-\dot w(t)\rvert 
  \leq
   \left( \sup_{\hat{c} \in \sS(c(0))}\sum_{r\in\sR}\bk\super{r}(\hat{c})\right)
    \underbrace{\int_0^t\!\lvert \e^{\zeta_\delta(\hat t)}-\e^{\zeta(\hat t)}\rvert\,d\hat t}_{\leq 4\delta \exp\left(\norm{\zeta}_{L^\infty}\right)}
  \e^{\Lip(\bk)\lVert\Gamma\rVert\e^{\lVert\zeta\rVert_{L^\infty}} t}
\label{eq:cutoff zeta estimate}
\end{equation}
and so $w_\delta\to w$ in $W^{1,\infty}(0,T;\RR^\sR_+)$, and by boundedness of the operator $\Gamma$ also $c_\delta\to c$ in $W^{1,\infty}(0,T;\RR^\sY_+)$.

For the convergence~\eqref{it:approx IV conv} we only need to prove convergence of the dynamic rate $\J$: the initial conditions are identical. Indeed, by dominated convergence together with \eqref{eq:cutoff zeta estimate} and $\zeta_\delta\leq\zeta$,
\begin{equation*}
  \J(c_\delta,w_\delta) = G(c_\delta,w_\delta,\zeta_\delta)\to G(c,w,\zeta)=\J(c_\delta,w_\delta).
\end{equation*}
\end{proof}

\begin{corollary} Let $\mu\super{V}$ satisfy Assumption~\ref{ass:cont initial ldp} and $\bk$ satisfy Assumptions~\ref{ass:rrate conditions}\eqref{it:rrate continuous}, \eqref{it:rrate upper bound}, \eqref{it:rrate monotonicity} and \eqref{it:rrate super-homogeneity}.
Given $(c,w)\in \BV\big(0,T;\RR^\sY_+\times\RR^\sR_+\big)$ such that $\J(c,w)<\infty$, there exists a sequence $(c_\delta,w_\delta)_\delta\subset\sA$ such that:
\begin{enumerate}[(i)]
\item $c_\delta(0)\to c(0)$ and $(c_\delta,w_\delta) \hybridto (c,w)$ as $\delta\to0$,
\item $\I_0(c_\delta(0))+\J(c_\delta,w_\delta)\to\I_0(c(0))+\J(c,w)$ as $\delta\to0$.
\end{enumerate}
\label{cor:approx}
\end{corollary}



\section{Large deviations}
\label{sec:ldp} 

We approach the proof of the main result, Theorem~\ref{th:ldp} with a fairly classical tilting approach with a twist. In Section~\ref{subsec:exponential tightness} we prove exponential tightness, in Section~\ref{subsec:lower bound} we prove the large deviations lower bound under initial distribution~$\mu\super{V}$, exploiting the approximation arguments from Section~\ref{subsec:approximations}. In Section~\ref{subsec:upper bound} we first prove the weak upper bound (i.e. on compact sets) for the conditional path measures, and then for the path measures under initial distribution~$\mu\super{V}$ again. The exponential tightness then guarantees that the lower bound also holds on closed sets, and that the rate functional is lower semicontinuous~\cite[Lem.~1.2.18]{Dembo1998}.

\subsection{Exponential tightness}
\label{subsec:exponential tightness}

By a standard Chernoff argument, the balls
\begin{align}
  \sB_m^\TV:=\Big\{ (c,w)\in\BV(0,T;\RR^\sY_+\times\RR^\sR_+) : \lVert(c,w)\rVert_{L^1} + \lVert (\dot c,\dot w)\rVert_\TV\leq m\Big\}
\end{align}
can be used for the exponential tightness. However, in order to control the initial condition in the large-deviations upper bound we work with the cones (for some initial condition $\tilde c(0)$):
\begin{equation*}
  \sC_{m,\epsilon}:=\Big\{( c, w)\in\sB_m^\TV: \big\lvert \big(c(t),w(t)\big) - \big(\tilde c(0),0)\big) \big\rvert \leq \epsilon + tm\quad t\text{--a.e.} \Big\},
\end{equation*}

\begin{lemma} For any $m,\epsilon>0$ the cone $\sC_{m,\epsilon}$ is hybrid-compact.
\end{lemma}
\begin{proof}
The cone $\sC_{m,\epsilon}$ is contained in the total-variation ball $\sB_m^\TV$ and is clearly $L^1$-bounded, so it is relatively compact as discussed in Section~\ref{subsec:BV}. We thus need to show that $\sC_{m,\epsilon}$ is hybrid closed. To that aim, take a hybrid-convergent net $(c\super{\omega},w\super{\omega})_{\omega} \subset \sC_{m,\epsilon}$ with limit $(c,w)$. By the weak-* lower semicontinuity of the $\TV$-norm it follows that $\lVert (\dot c,\dot w)\rVert_{\TV}\leq m$. Moreover, the pointwise bound implies that,
\begin{align*}
  &\int_A\!\Big(\big\lvert(c\super{\omega}(t),w\super{\omega}(t))-(\tilde c(0),0)\big\rvert - \epsilon-mt\Big)\,\dd t \leq 0 & \forall\text{ measurable } A\subset(0,T),
  \intertext{hence, after taking the limit in $\omega$,}
  &\int_A\!\Big(\big\lvert(c(t),w(t))-(\tilde c(0),0)\big\rvert - \epsilon-mt\Big)\,\dd t \leq 0 & \forall\text{ measurable } A\subset(0,T),
\end{align*}
which is equivalent to the pointwise bound $\lvert(c(t),w(t))-(\tilde c(0),0)\rvert \leq \epsilon+mt$ for the limit.
\end{proof}

We first show exponential tightness for the conditional measures.
\begin{lemma}[Uniform Exponential tightness of conditional measures]
Let $\zeta \in C_\mathrm{c}(0,T; \RR^\sR)$ and assume $\bk\super{r}$ is bounded on stoichiometric simplices (Assumption~\ref{ass:rrate conditions}\eqref{it:rrate upper bound}). Fix any convergent sequence $\tfrac1V\NN_0^\sR\ni\tilde c\super{V}(0)\to\tilde c(0)\in\RR^\sR_+$ and let $\tilde\PP_\zeta\super{V}$ be the law of the Markov process with generator $\sQ_{\zeta, t}\super{V}$ and initial distribution $\delta_{c\super{V}(0)}$. Then for any $\epsilon$ and $\eta>0$ there exists an $m$ (not depending on the choice $\tilde c(0)$) such that
\begin{equation*}
  \mfrac1V\log\tilde\PP_\zeta\super{V}\big(\sC_{m,\epsilon}^\mathsf{c}\big) \leq -\eta.
\end{equation*}
\end{lemma}

\begin{proof}
For a $\delta>0$ to be determined later, define the set (see Figure~\ref{fig:discretised cone}):
\begin{align*}
  \Sigma_{\delta,\epsilon}
    &:=\Big\{ 
   (c, w)\in \BV(0,T;\RR^\sY\times\RR^\sR_+):
    \big\lvert (c(t),w(t)) - (\tilde c(0),0)\big\rvert\leq \sigma_{\delta,\epsilon}(t)\quad t-\ae \Big\},\\
  \sigma_{\delta,\epsilon}(t)
    &:= \epsilon + \sum_{l=1}^{\lfloor T/\delta\rfloor} \tfrac12\epsilon l\mathds1_{\lbrack l\delta,(l+1)\delta)}(t).
\end{align*}
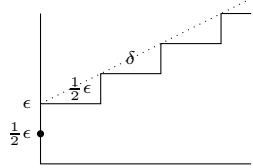
\begin{figure}[h!]
\centering
\begin{tikzpicture}[scale=0.8]
  \tikzstyle{every node}=[font=\scriptsize]
  \draw(0,2.5)--(0,0)--(3.5,0);
  \filldraw(0,0.5) node[anchor=east]{$\tfrac12\epsilon$} circle(0.05);
  \draw(0,1) node[anchor=east]{$\epsilon$} -- (1,1)-- node[midway,anchor=east]{$\tfrac12\epsilon$} (1,1.5) -- node[midway,anchor=south]{$\delta$} (2,1.5)--(2,2)--(3,2)--(3,2.5)--(3.5,2.5);
  \draw[dotted](0,1)--(3.5,2.75);
\end{tikzpicture}
\caption{the function~$\sigma_{\delta,\epsilon}(t)\leq\epsilon + \tfrac\epsilon{2\delta}t$.}
\label{fig:discretised cone}
\end{figure}
Then $\sC_{\epsilon/(2\delta),\epsilon} \supset \Sigma_{\delta,\epsilon}\cap \sB^\TV_{\epsilon/(2\delta)}$ and so it suffices to prove that for any $\eta>0$ we can find $m,\delta>0$ such that
\begin{align}
  \limsup_{V\to\infty}\tfrac1n\log\tilde\PP\super{V}(\Sigma_{\delta,\epsilon}^\mathsf{c})&\leq -\eta \qquad{and}
\label{eq:discr cone exp tight}\\
  \limsup_{V\to\infty}\tfrac1n\log\tilde\PP\super{V}({\sB^\TV_m}^\mathsf{c})&\leq -\eta.
\label{eq:TV ball exp tight}
\end{align}

Observe that by the convergence of the initial condition, for $V$ sufficiently large and $\tilde\PP\super{V}$-almost surely,
\begin{equation}
  \lvert(c(0),w(0))-(\tilde c(0),0)\rvert \leq \tfrac12\epsilon.
\label{eq:init cond arbitrary close}
\end{equation}
To prove \eqref{eq:TV ball exp tight}, observe that the Markov jump process $\sum_{r\in\sR}W\super{V,r}(t)$ is bounded by a Poisson process $\tfrac1V N_{V\lambda}(t)$ with $\lambda := e^{\norm{\Gamma}\norm{\zeta}_\infty}\sup_{\sS_{1/2\epsilon}(\tilde c(0))} \sum_{r\in\sR} 1+\bk\super{r}<\infty$ due to Assumptions~\ref{ass:rrate conditions}\eqref{it:rrate convergence} and \eqref{it:rrate upper bound}. A standard Chernoff bound therefore yields
\begin{align*}
  \tilde\PP\super{V}({\sB^\TV_m}^\mathsf{c}) &= \tilde\PP\super{V}\big(\big\{ {\textstyle\sum_{r\in\sR} W\super{V,r}(T)>m} \big\}\big) \\
                                             &\leq \Prob\big( N_{V\lambda}(T)> Vm \big) 
                                                        \leq e^{V\lambda T V - nm} 
                                                        = e^{-V\eta},
\end{align*}
if we choose $m:=\lambda Te+\eta$.

We now prove \eqref{eq:discr cone exp tight}. Because of \eqref{eq:init cond arbitrary close} we may assume that for any $(c,w)\in \Sigma_{\delta,\epsilon}^\mathsf{c}$ there exists an interval $(l\delta,(l+1)\delta)$ on which the process has jumped more than $\tfrac12\epsilon$. Since the norm of each jump is bounded from below by $\tfrac1V$ (the $W$-coordinate always jumps at least that length) we can estimate:
\begin{align*}
  \tilde\PP\super{V}(\Sigma_{\delta,\epsilon}^\mathsf{c})
  &\leq \tilde\PP\super{V}\Big({\textstyle \bigcup_{l=1}^{\lfloor T/\delta\rfloor} \big\{ \tfrac1V \sum_{r\in\sR} VW\super{V,r}((l+1)\delta) - VW\super{V,r}(l\delta)>\tfrac\epsilon2\big\}}\Big) \\
  & \leq \sum_{l=1}^{\lfloor T/\delta\rfloor} \Big({\textstyle \big\{ \tfrac1V \sum_{r\in\sR} VW\super{V,r}((l+1)\delta) - VW\super{V,r}(l\delta)>\tfrac\epsilon2\big\}}\Big) \\
  &\leq \mfrac{T}{\delta}\Prob\big( N_{V\lambda}(\delta) > \tfrac{V\epsilon}2\big) \\
  &\leq \mfrac{T}{\delta} e^{-V s(\epsilon/2 \vert \lambda\delta)},
\end{align*}
where the latter is found by first applying a Chernoff bound to $\Prob\big( aN_{V\lambda}(\delta) > \tfrac{aV \epsilon}2\big)$ for arbitrary $a>0$ and then minimising over $a$. With the choice
\begin{equation*}
  \delta := \mfrac{\epsilon}{2\lambda}\exp\big( - \mfrac{2\eta}{\epsilon} -1 \big),
\end{equation*}
we find $s(\epsilon/2 \vert \lambda\delta) = \eta + \tfrac{\epsilon}{2}e^{-2\eta/\epsilon-1} \geq \eta$ which proves \eqref{eq:discr cone exp tight}.
\end{proof}

From \cite[Prop.~6]{Biggins2004} we now immediately obtain exponential tightness under the initial distribution $\mu\super{V}$:
\begin{corollary}[Exponential tightness]
Let $\zeta \in C_\mathrm{c}(0,T; \RR^\sR)$, assume $\bk\super{r}$ is bounded on stoichiometric simplices (Assumption~\ref{ass:rrate conditions}\eqref{it:rrate upper bound} and let $\mu\super{V}$ be exponentially tight (Assumption~\ref{ass:cont initial ldp}\eqref{it:initial exp tight}). Then $\PP_\zeta\super{V}$ (under initial distribution $\mu\super{V}$) is exponentially tight.
\label{prop:exp tight cone}
\end{corollary}

\begin{remark} The results in this section apply in particular to the $\PP\super{V} = \PP_0\super{V}$.
\end{remark}

\subsection{Lower bound}
\label{subsec:lower bound}

\begin{proposition} Let $\mu\super{V}$ satisfy Assumption~\ref{ass:cont initial ldp} and $\bk$ satisfy Assumptions~\ref{ass:rrate conditions}. For any hybrid-open set $\sO\subset\BV\big(0,T;\RR^\sY\times\RR^\sR\big)$,
\begin{equation*}
  \liminf_{V\to\infty} \frac1V \log \PP\super{V}(\sO) \geq - \inf_{(c,w)\in\sO} \I_0(c(0))+\J(c,w).
\end{equation*}
\label{prop:lower bound}
\end{proposition}

\begin{proof} Recall the definition of the set $\sA$ in \eqref{eq:suff regular curves for Girsanov}. Choose an arbitrary hybrid-open set $\sO\subset\BV(0,T;\RR^\sY\times\RR^\sR)$. From Lemma~\ref{lem:lower bound for suff reg curves} proven below, it follows that
\begin{equation*}
  \liminf_{V\to\infty} \frac1V \log \PP\super{V}(\sO) \geq - \inf_{(c,w)\in\sO\cap\sA} \I_0(c(0))+\J(c,w).
\end{equation*}
By Corollary~\ref{cor:approx} it then follows that
\begin{equation*}
  \inf_{(c,w)\in\sO} \I_0(c(0))+\J(c,w)=\inf_{(c,w)\in\sO\cap\sA} \I_0(c(0))+\J(c,w).
\end{equation*}

\end{proof}

For the lower bound it thus remains to prove the following:

\begin{lemma} Let Assumption~\ref{ass:rrate conditions} on the rates and Assumption~\ref{ass:cont initial ldp} on the initial distribution hold. Let $\sO\subset\BV\big(0,T;\RR^\sY\times\RR^\sR\big)$ be any hybrid-open set, and $\sA$ be the set~\eqref{eq:suff regular curves for Girsanov}. Then for any $(c,w)\in \sO\cap \sA$,
\begin{equation}
  \liminf_{V\to\infty} \frac1V \log \PP\super{V}(\sO) \geq - \I_0\big(c(0)\big) - \J(c,w).
\label{eq:lower bound regular curves}
\end{equation}
\label{lem:lower bound for suff reg curves}
\end{lemma}

\begin{proof}
Take a pair $(c,w)\in \sO\cap \sA$, and let $\zeta:=\log\dot w/\bk(c) \in C_c^1\big(0,T;\RR^\sR\big)$ and $z\in\partial\I_0\big(c(0)\big)$, which is non-empty by Assumption~\ref{ass:cont initial ldp}. Without loss of generality we assume that $\I_0(c(0))+\J(c,w)<\infty$. 

We also define a perturbed initial distribution on $\RR^\sY$ by setting:
\begin{equation*}
  \mu\super{V}_z\big(d\tilde c(0)\big) =  \e^{V z\cdot \tilde c(0) - V\Lambda\super{V}(z)}\mu\super{V}\big(d\tilde c(0)\big), \qquad\text{with}\quad \Lambda\super{V}(z):=\tfrac1V\log\int\!\e^{V z\cdot\tilde c(0)}\mu\super{V}\big(d\tilde c(0)\big).
\end{equation*}
By Assumption~\ref{ass:cont initial ldp}, we can apply Varadhan's Lemma~\cite[Th.~4.3.1]{Dembo1998}, and so, combined with the assumption $z\in\partial\I_0\big(c(0)\big)$,
\begin{equation}
  \lim_{V\to\infty} \Lambda\super{V}(z) = \sup_{\tilde c(0)\in\RR^\sY_+} z\cdot\tilde c(0) - \I_0\big(\tilde c(0)\big) = z\cdot c(0) - \I_0\big(c(0)\big)=:\Lambda(z),
\label{eq:initial Varadhan}
\end{equation}
Since we assumed that $\I_0\big(c(0)\big)$ is finite, it follows that (at least for sufficiently large $V$), the value $\Lambda\super{V}(z)$ is finite. Naturally $\e^{V\Lambda\super{V}(z)}$ is simply a normalisation factor so that the perturbed $\mu\super{V}_z$ is a probability measure. We can now define the perturbed path measure
\begin{equation*}
  \PP_{\zeta,z}\super{V}(\dd w^\prime\, \dd c^\prime):=\int\!\PP_{\zeta}\super{V}\big(\dd w^\prime\,\dd c^\prime \mid c^\prime(0) = \tilde c(0)\big) \mu_z\super{V}\big(d\tilde c(0)\big).
\end{equation*}
The next step is to apply Theorem~\ref{th:change of measure general} to see that
\begin{equation}\label{eq:change of measure}
  \log\frac{\dd \PP\super{V}}{\dd \PP_{\zeta,z}\super{V}}\left(c^\prime,w^\prime\right)=
  -VG\super{V}(c^\prime, w^\prime, \zeta) -V z\cdot c^\prime(0) + V\Lambda\super{V}(z).
\end{equation}
When checking the applicability of the results from the appendix one may take
$K_{(c^\prime,w^\prime)} = \sS(c^\prime) \times \RR_+^\sR $.
To establish Assumption~\ref{it:change-of-measure ass 3} one observes that $\sup_{t\in(0,T)} \abs{c^\prime(t),w^\prime(t)}$ is bounded by $\abs{c^\prime(0)}$ plus a constant times the number of jumps up to time $T$, and that  under $\PP\super{V}$ the number of jumps is stochastically dominated by a Poisson random variable with finite expectation due to Assumption~\ref{ass:rrate conditions} parts~(\ref{it:rrate convergence})\&(\ref{it:rrate upper bound}). 

We now apply a standard tilting argument with respect to this measure. We first introduce the sets, for some arbitrary small $\epsilon>0$ (recall that $(c,w)$ is already fixed),
\begin{align}
  &\sG_\epsilon^\zeta=\sG_\epsilon^\zeta\lbrack c,w\rbrack := \Big\{(c^\prime,w^\prime) \in \BV\left(0,T;\RR^\sY\times\RR^\sR\right) \colon \big\lvert G(c^\prime,w^\prime,\zeta)-G(c,w,\zeta)\big\rvert<\epsilon \Big\},\quad\text{and}
  \label{eq:set G-ball}
  \\
  &\sB_\epsilon=\sB_\epsilon\lbrack c(0)\rbrack:=\big\{c^\prime(0)\in\RR^\sY_+: \lvert c^\prime(0)-c(0)\rvert<\epsilon \big\},
  \label{eq:set c0-ball}
\end{align}
and let $\pi_0\lbrack c^\prime\rbrack:=c^\prime(0)$.
Although $\sG_\epsilon^\zeta$ is not restricted to the positive cone, the probabilities are of course concentrated on non-negative concentrations and fluxes. Using \eqref{eq:change of measure}
\begin{align}
  &\frac1V\log\PP\super{V}\big(\sO\big) 
 \geq 
   \inf_{\substack{(c^\prime,w^\prime)\in \\ \sO\cap \sG_\epsilon^\zeta \cap \pi_0^{-1}\lbrack\sB_\epsilon\rbrack }}\Big\lbrack - z\cdot c^\prime(0) + \Lambda\super{V}(z) - G\super{V}(c^\prime,w^\prime,\zeta) \Big\rbrack
        + \frac1V\log\PP\super{V}_{\zeta,z}\big(\sO \cap \sG_\epsilon^\zeta \big),
\label{eq:lower bound estimate 1}
\intertext{where}
  &G\super{V}(c^\prime,w^\prime,\zeta):=\int_0^T\!\big\lbrack \zeta(t)\cdot \dot w^\prime(\dd t) - \sum_{r\in\sR}\tfrac1V k\super{V,r}\big(c^\prime(t)\big)\big(e^{\zeta\super{r}}(t)-1\big) \big\rbrack\,\dd t,
\label{eq:GV}
\end{align}
The first term is bounded by $-z\cdot\tilde c(0)\geq -z\cdot c(0) - \epsilon$ by definition of $\sB_\epsilon$; for the second term we use~\eqref{eq:initial Varadhan} so that
\begin{equation}
  \abs{\Lambda\super{V}(z) - \Lambda(z)} < \epsilon.
\label{eq:LambdaV bound}
\end{equation}
for $V$ sufficiently large. For the third term we estimate,
\begin{multline}
  \big\lvert G\super{V}(c^\prime, w^\prime,\zeta) - G(c,w,\zeta) \big\rvert
  \leq  \\
    \sup_{(c^\prime, w^\prime)\in \pi_0^{-1}\lbrack \sB_\epsilon\rbrack} 
      \abs{G\super{V}(c^\prime, w^\prime,\zeta)-G(c^\prime, w^\prime,\zeta)} 
      + \sup_{(c^\prime, w^\prime)\in  \sG_\epsilon^\zeta} \abs{G(c^\prime, w^\prime,\zeta)-G(c,w,\zeta)}\\
  \leq T(\e^{\lVert\zeta\rVert_{L^\infty}}+1) \sup_{ c^\prime \in \sS_\epsilon(c(0))} \sum_{r\in\sR}\lvert\tfrac1Vk\super{V,r}(c^\prime)-\bk\super{r}(c^\prime)\rvert
    + \epsilon \\
   \leq \const \epsilon,
\label{eq:GV bound}
\end{multline}
for sufficiently large $V$ because of Assumption~\ref{ass:rrate conditions}\eqref{it:rrate convergence}. For the last term in \eqref{eq:lower bound estimate 1} we use Proposition~\ref{th:unperturbed convergence det} together with the Portemanteau Theorem:
\begin{align*}
  \liminf_{V\to\infty} \PP\super{V}_{\zeta,z}\big(\sO \cap \sG_\epsilon^\zeta \cap \pi_0^{-1}\lbrack \sB_\epsilon\rbrack \big)
    &\geq 
  \liminf_{V\to\infty} \PP\super{V}_{\zeta,z}\big(\sO \cap \sG_\epsilon^\zeta \big)- \limsup_{V\to\infty}\PP\super{V}_{\zeta,z}\big( \pi_0^{-1}\lbrack \sB_\epsilon\rbrack^\mathsf{c} \big)\\
    &=
\liminf_{V\to\infty} \PP\super{V}_{\zeta,z}\big(\sO \cap \sG_\epsilon^\zeta \big)-\limsup_{V\to\infty} \mu\super{V}_{z}\big( \sB_\epsilon^\mathsf{c} \big) \geq 1,
\end{align*}
which is valid since $\sO \cap \sG_\epsilon^\zeta$ is hybrid-open by the continuity of $(c,w)\mapsto G(c,w,\zeta)$ (recall $\zeta\in C_c^1(0,T;\RR^\sR)$, and $\sB_\epsilon^\mathsf{c}$ is closed in $\RR^\sY$.

Putting all these estimates and convergence results together we find from \eqref{eq:lower bound estimate 1} that
\begin{align*}
  \liminf_{V\to\infty} \frac1V\log\PP\super{V}\big(\sO\big)\geq - z\cdot c(0) + \Lambda\super{V}(z) - G(c,w,\zeta) - \mathit{const}\,\epsilon = -\I_0\big(c(0)\big) - \J(c,w) - \mathit{const}\,\epsilon,
\end{align*}
as $\zeta$ and $z$ were chosen to make the final equality true, assuming convexity of $\I_0$. This proves the claim since $\epsilon$ was arbitrary.
\end{proof}

\subsection{Upper bound}
\label{subsec:upper bound}

For the upper bound we work first with a deterministic initial condition and then an argument of Biggins'  \cite{Biggins2004} to deduce the upper bound for the `mixture'.

\begin{lemma} Let $\bk$ satisfy Assumptions~\ref{ass:rrate conditions}\eqref{it:rrate continuous},\eqref{it:rrate convergence},\eqref{it:rrate upper bound}, and fix any convergent sequence $\tilde c\super{V}(0)\to \tilde c(0)$ in $\RR_+^\sY$.
Let $\tilde\PP\super{V}$ be the law of the Markov process with deterministic initial condition $\tilde c\super{V}(0)$ and the dynamics given by \eqref{eq:generator}.
Then for any hybrid-compact set $\sK\subset\BV\big(0,T;\RR^\sY\times\RR^\sR\big)$,
\begin{equation*}
  \limsup_{V\to\infty} \frac1V \log \tilde\PP\super{V}(\sK)
     \leq - \inf_{\substack{(c,w)\in\sK \\ c(0)=\tilde c(0)}} \J(c,w).
\end{equation*}
\label{lem:upper bound dynamic}
\end{lemma}

\begin{proof}
We use an adaptation of the usual covering technique as in the proof of the G{\"a}rtner-Ellis Theorem~\cite[Th.~4.5.3]{Dembo1998}. Fix a convergent sequence $\tilde c\super{V}(0)\to \tilde c(0)$ in $\RR_+^\sY$, a hybrid-compact set $\sK\subset\BV\big(0,T;\RR^\sY\times\RR^\sR\big)$, and arbitrary $\epsilon>0$. 

To control the initial condition, we use the compact cones $\sC_{m,\epsilon}$ from the exponential tightness, Proposition~\ref{prop:exp tight cone}. In this proof we will take $m>0$ such that $\limsup_{V\to\infty}\frac1V\log\tilde\PP\super{V}\big( \sC_{m,\epsilon}^\mathsf{c})<-1/\epsilon$. Note that for $V$ sufficiently large, $\tilde c\super{V}(0)\in\sB_\epsilon\lbrack\tilde c(0)\rbrack=\pi_0\sC_{m,\epsilon}$. 

By Proposition~\ref{prop:J=tilde J} we can find, for any $(c,w)\in\BV\big(0,T;\RR^\sY\times\RR^\sR\big)$ and $\epsilon>0$, a $\zeta\lbrack c,w\rbrack\in C_c^1(0,T;\RR^\sR)$ such that $G(c,w,\zeta\lbrack c,w\rbrack)\geq\J(c,w) -\epsilon$. Then the sets $\sG_\epsilon^\zeta\lbrack c,w\rbrack$ from~\eqref{eq:set G-ball} form an open covering $\bigcup_{(c,w)\in\sK} \sG_\epsilon^{\zeta\lbrack c,w\rbrack}(c,w) \supset \sK\cap\sC_{m,\epsilon}$, and hence there exists a finite subset $(c\super{n},w\super{n})_{n=1,\hdots,N}\subset\sK$ such that $\bigcup_{n=1,\hdots,N} \sG_\epsilon^{\zeta\lbrack c\super{n},w\super{n}\rbrack}(c\super{n},w\super{n}) \supset \sK\cap\sC_{m,\epsilon}$.

Let $\tilde\PP\super{V}_\zeta$ be the law of the Markov process with deterministic initial condition $\tilde c\super{V}(0)$ and the dynamics given by the perturbed generator \eqref{eq:perturbed generator}
For each $n=1,\hdots,N$ we find for sufficiently large $V$ (here we only intersect with the initial balls in order to employ~\eqref{eq:GV bound} in the end of this calculation),
\begin{align*}
    &\mfrac1V \log \tilde\PP\super{V}\big(\sG^{\zeta\lbrack c\super{n},w\super{n}\rbrack}_\epsilon(c\super{n},w\super{n}) \big) 
  = 
     \mfrac1V \log \tilde\PP\super{V}\big(\sG^{\zeta\lbrack c\super{n},w\super{n}\rbrack}_\epsilon(c\super{n},w\super{n})
      \cap \pi_0^{-1}\sB_\epsilon\lbrack\tilde c(0)\rbrack \big) \\
   &\leq
    \sup_{(c,w)\in \sG^{\zeta\lbrack c\super{n},w\super{n}\rbrack}_\epsilon(c\super{n},w\super{n}) 
        \cap \pi_0^{-1}\sB_\epsilon\lbrack\tilde c(0)\rbrack } 
        \mfrac1V\log\frac{d\tilde\PP\super{V}}{d\tilde\PP_{\zeta\lbrack c\super{n},w\super{n}\rbrack}\super{V}}(c,w)  \\
   &\hspace{8cm} + \underbrace{\mfrac1V\log \tilde\PP_{\zeta\lbrack c\super{n},w\super{n}\rbrack}\super{V}\big(  \sG^{\zeta\lbrack c\super{n},w\super{n}\rbrack}_\epsilon(c\super{n},w\super{n}) \big)}_{\leq0} \\
  &\stackrel{\text{Th.~\ref{th:change of measure general}}}{\leq}
   \sup_{(c,w)\in \sG^{\zeta\lbrack c\super{n},w\super{n}\rbrack}_\epsilon(c\super{n},w\super{n})
    \cap \pi_0^{-1}\sB_\epsilon\lbrack\tilde c(0)\rbrack } 
    - G\super{V}(c,w,\zeta\lbrack c\super{n},w\super{n}\rbrack) \\
  &\,\,\,\stackrel{\eqref{eq:GV bound}}{\leq}
   \sup_{(c,w)\in \sG^{\zeta\lbrack c\super{n},w\super{n}\rbrack}_\epsilon(c\super{n},w\super{n}) } 
      - G(c,w,\zeta\lbrack c\super{n},w\super{n}\rbrack) + \const\epsilon \\
  &\,\,\,\stackrel{\eqref{eq:set G-ball}}{\leq} - G(c\super{n},w\super{n},\zeta\lbrack c\super{n},w\super{n}\rbrack) + \const\epsilon.
\end{align*}

Because of the finiteness of the covering we can now use the Laplace Principle:
\begin{align*}
  \limsup_{V\to\infty}\mfrac1V\log \tilde\PP\super{V}(\sK) 
    &\leq \limsup_{V\to\infty}\mfrac1V\log \big(\tilde\PP\super{V}(\sK\cap\sC_{m,\epsilon}) + \tilde\PP\super{V}(\sC_{m,\epsilon}^\mathsf{c})\big)\\
    &\leq \max_{n=1,\hdots,N} \limsup_{V\to\infty}\mfrac1V\log\tilde\PP\super{V}(\sG^{\zeta\lbrack c\super{n},w\super{n}\rbrack}_\epsilon) \vee -\mfrac1\epsilon\\
    &\leq \max_{n=1,\hdots,N} \big(- G(c\super{n},w\super{n},\zeta\lbrack c\super{n},w\super{n}\rbrack) + \const\epsilon\big) \vee -\mfrac1\epsilon\\
    &\leq \max_{n=1,\hdots,N} \big(- \J(c\super{n},w\super{n}) + \const\epsilon\big) \vee -\mfrac1\epsilon\\
    &\leq \Big(- \inf_{(c,w)\in\sK\cap\sC_{m,\epsilon}} \J(c,w) + \const\epsilon\Big)\vee-\mfrac1\epsilon\\
    &\leq \Big(- \inf_{\substack{(c,w)\in\sK:\\c(0)\in\sB_\epsilon\lbrack\tilde c(0)\rbrack}} \J(c,w) + \const\epsilon\Big)\vee-\mfrac1\epsilon.
\end{align*}
This proves the claim as $\epsilon$ was chosen arbitrarily.
\end{proof}

We can now deduce the large-deviations upper bound for the mixture:

\begin{corollary}

Let $\mu\super{V}$ satisfy Assumption~\ref{ass:cont initial ldp} and $\bk$ satisfy Assumptions~\ref{ass:rrate conditions}. For any hybrid-compact set $\sK\subset\BV\big(0,T;\RR^\sY\times\RR^\sR\big)$,
\begin{equation*}
  \liminf_{V\to\infty} \frac1V \log \PP\super{V}(\sK) \geq - \inf_{(c,w)\in\sK} \I_0(c(0))+\J(c,w).
\end{equation*}
\label{prop:upper bound}
\end{corollary}
\begin{proof} By Assumption~\ref{ass:cont initial ldp} and Lemma~\ref{lem:upper bound dynamic} one can apply \cite[Lemma~12]{Biggins2004} noting that the proof in \cite{Biggins2004} only uses the upper bound proved in Lemma~\ref{lem:upper bound dynamic} not a full LDP.
\end{proof}

\appendix
\section{A change-of-measure result for linear test functionals on jump processes}

Changes of measure are central to the proof of the large deviations principle presented in this work.
This appendix arose out of the need to clarify under exactly what technical conditions  \cite[Appendix~1, Prop.~7.3]{Kipnis1999} could be adapted to the setting of the present work, in particular so that functions of the form $x \mapsto \zeta\cdot x$ could be used since these are not bounded functions (although they are bounded linear operators).
This boundedness restriction is avoided in \cite{Palmowski2002}, but functions used in the change of measure are no longer time dependent and the conditions are less explicit.
Here the aim is to include unbounded, time dependent functions in the change of measure formula, but to give relatively explicit, sufficient conditions that can easily be checked using the model assumptions from the main part of the paper.
In this endeavour the results are restricted to pure jump processes.

Let $\X$ be a Banach space, $T \in (0,\infty]$ and $\left(\Omega, \sF, (\sF_t)_{t\in[0,T)}\right)$ be a filtered probability space with canonical random variable $X:\Omega\to\Omega$, where
\begin{itemize}
\item $\Omega$ is a subset of the c\`adl\`ag functions $[0,T) \rightarrow \X$, with the convention $f(T):=f(T-)$ if $T<\infty$,
\item $\sF$ is the Borel $\sigma$-algebra generated by a separable topology on $\Omega$ and equal to the $\sigma$-algebra generated by the time evaluation functions $X \mapsto X(t)$.
\end{itemize}
Note that $T=\infty$ is allowed for now. The application in this paper is to the case $\Omega = \BV(0,T, \RR^\sY \times \RR^\sR)$ with the hybrid topology, but this is not a necessary assumption.

We define the jump process through a given family of jump kernels $(\alpha_t(x,\cdot)_{t\in[0,T),x\in\X}$ where $\alpha_t(x,A)$ is the instantaneous jump rate at time $t$ from $x\in\X$ into a measurable set $A\subset\X$, together with a given initial distribution $\mu$. Let $\PP$ be the law of this process, a probability measure on $(\Omega, \sF)$ and $\EE$ the associated expectation operator.

We now define a class of test functions for which the associated propagators (a two-paramater semigroup of linear operators) are well-defined. To construct this set we will assume that there exists a family of measurable (not necessarily compact or bounded) subsets $(K_x)_x$ of $\X$ such that for all $x\in\X$:
\begin{itemize}
\item $x\in K_x$ and $\bigcup_{y\in K_x}K_y = K_x$,
\item $\int_0^T\sup_{y \in K_x}\alpha_t(y,\X)\dd t < \infty$,
\item $\sup_{t\in[0,T)}\sup_{y \in K_x}\alpha_t(y,\X \setminus K_x) = 0$.
\end{itemize}
This expresses the idea that the process started from $x$ can never explode nor leave $K_x$. Then the propagators $(P_{s,t}f)(x) := \EE\left[f(X(t)) \middle\vert X(s)= x \right]$
preserve the set
\begin{equation*}
  B_\mathrm{K}(\X) := \left\{f :\X\rightarrow \RR \text{ measurable, such that } \forall x \in \X \sup_{y\in K_x}\abs{f(y)}< \infty \right\},
\end{equation*}
and satisfy $\frac{\dd}{\dd s}(P_{s,t}f)(x) = -(\sQ_s P_{s,t}f)(x)$ with (time-dependent) generator
\begin{align*}
  (\sQ_t f)(x) := \int_X \left[f(y)-f(x) \right]\alpha_t(x,\dd y).
\end{align*}

We now make three additional assumptions under which the change-of-measure formula holds.
\begin{itemize}
\item there is a $\gamma > 0$ such that $\abs{y-x} \leq \gamma$ for all $x\in\X,t \in (0,T)$, and $\alpha_t(x,\cdot)$-\ae\, $y$,
  \eqnum\label{it:change-of-measure ass 2}
\item $\lim_{n \rightarrow \infty} \PP\left(\tau_n < t\right) = 0$ for all $t\in(0,T), n \in \NN$, where $\tau_n := \inf \left\{t \colon \alpha_t(X(t),\X ) \geq n \right\}$,
  \eqnum\label{it:change-of-measure ass 1}
\item $\mathbb{E}\left[Z^\beta(t)\right]< \infty$ for all $t\in(0,T),\beta>0$, where
  \eqnum\label{it:change-of-measure ass 3}
\begin{multline*}
Z^\beta(t) :=
\exp\left(\beta \abs{X(0)}\right) + \exp\left(\beta \abs{X(t)}\right)\\
+ \int_0^t\!\exp\left(\beta \abs{X(s)}\right)\beta  \abs{X(s)}\,\dd s
+ \int_0^t\!\exp\left(\beta \abs{X(s)}\right)\beta \alpha_s\!\left(X(s),\X\right)\,\dd s.
\end{multline*}
\end{itemize}

The next result is a variation on \cite[Appendix~1, Lem.~5.1]{Kipnis1999}:
\begin{proposition}\label{prop:bmart}
Let $f \colon [0,T) \times \X \rightarrow \RR$ be bounded, absolutely continuous in $t$ and measurable in $x$, with measurable, uniformly bounded derivative $\partial_t f(t,x)$. Then under Assumptions~\eqref{it:change-of-measure ass 1} \& \eqref{it:change-of-measure ass 2},
\end{proposition}
\begin{equation*}
  M^f(t) := f\left(t,X(t)\right) - f\left(0,X(0)\right) 
   - \int_0^t\!\bigl(\left(\partial_s + \sQ_s\right)f\bigr)\left(s,X(s)\right)\,\dd s
\end{equation*}
is a Martingale in the filtration $\left(\mathcal{F}_t\right)_{t\geq 0}$ generated by $X(t)$.
\begin{proof}
In the case that $f$ does not depend on time and $\sup_{t,x} \alpha_t(x,\X) < \infty $ the result follows from \cite[Ch.~4 Sect.~7]{Ethier86}. The additional term $\partial_s$ is added for time-dependent test functions due to a chain rule. By approximating by the process stopped at $\tau_n$ and using Assumptions~\eqref{it:change-of-measure ass 1}\&\eqref{it:change-of-measure ass 3} one can remove the boundedness assumption on $\alpha$.

\end{proof}

\begin{lemma}\label{lem:Martingale linear test function}
Under Assumptions~\eqref{it:change-of-measure ass 2}, \eqref{it:change-of-measure ass 1} \& \eqref{it:change-of-measure ass 3}, 
the conclusion of Proposition~\ref{prop:bmart} is valid when  $f(t,x) = \zeta(t) \cdot x$ and when $f(t,x) = e^{\zeta(t) \cdot x}$, in both cases for $\zeta \in C^1_\mathrm{b}\left([0,T);\X^\ast\right)$ where $\X^\ast$ is the Banach dual of $\X$.
\end{lemma}
\begin{proof}
The exponential case is proved here; the linear case is similar.
Let $\theta_n \in C^\infty(\RR)$ be such that $\theta_n(y) = y$ for $y\leq n$, $\theta_n \leq n+1$ and
$0\leq\theta_n^\prime\leq 1$. Take an arbitrary $\zeta \in C^1_\mathrm{b}\left([0,\infty);\X^\ast\right)$ and set $f_n(t,x) = \exp\!\left(\theta_n\!\left(\zeta(t)\cdot x \right)\right)$ so that Proposition~\ref{prop:bmart} can be applied to $M^{f_n}(t)$.

It follows from the definitions that for all $t$ and $x$,
\begin{equation*}
  \lim_n f_n(t,x) = f(t,x) \quad \text{and} \quad \lim_n \partial_t f_n(t,x) = \partial_t f(t,x).
\end{equation*}
Because of Assumption~\eqref{it:change-of-measure ass 2} $\sQ_t f$ is well defined and one can prove by dominated convergence that $\lim_n (\sQ_t f_n)(t,x)=(\sQ_t f)(t,x)$ for all $t,x$.  Preparatory to further applications of dominated convergence we estimate
\begin{align*}
  f_n(t,x) &\leq \exp\left(\norm{\zeta}_\infty \abs{x}\right), \\
  \abs{\partial_t f_n(t,x)} &\leq \exp\left(\norm{\zeta}_\infty \abs{x}\right)\norm{\dot \zeta}_\infty\abs{x}, \qquad\text{and}\\
  \abs{(\sQ_t f_n)(t,x)} &\leq \exp\left(\norm{\zeta}_\infty \abs{x}\right)
   \left(\exp\left(\norm{\zeta}_\infty \gamma\right)+1\right)\alpha_t(x,\X).
\end{align*}
With these estimates and Assumption~\eqref{it:change-of-measure ass 3} one checks $\lim_n M^{f_n}(t) = M^f(t)$ almost surely.
Again using Assumption~\eqref{it:change-of-measure ass 3} one can find a $\beta > 0$ such that $\abs{M^{f_n}(t)} \leq Z^\beta(t)$ almost surely.
By the conditional expectation form of the dominated convergence theorem, for $s<t$,
\begin{equation*}
  M^f(s)=\lim_n M^{f_n}(s) = \lim_n \EE\big\lbrack M^{f_n}(t) \vert \sF_s\big\rbrack = \EE\big\lbrack\lim_n M^{f_n}(t) \vert \sF_s\big\rbrack = \EE\big\lbrack M^f(t) \vert \sF_s\big\rbrack.
\end{equation*}
\end{proof}

Finally, for the exponential change of measure we will need a bounded time interval.
\begin{theorem} 
Let $T<\infty$, $\zeta\in C_b^1(0,T;\RR^\sR)$, and let Assumptions~\eqref{it:change-of-measure ass 2}, \eqref{it:change-of-measure ass 1} and \eqref{it:change-of-measure ass 3} all hold. Suppose $\PP_\zeta$ is the law of some process with paths in $\Omega$ and having initial distribution $\mu$. Under $\PP_\zeta$, $X$ is a Markov process with generator
\begin{equation*}
  (\sQ_{\zeta,t}f)(x)= \int_\X \left[f(y)-f(x) \right] \e^{\zeta(t)\cdot y - \zeta(t)\cdot x}\alpha_t(x,\dd y)
\end{equation*}
if and only if
\begin{equation}
  \log\frac{d\PP_\zeta}{d\PP}(X)=
    \zeta(T)\cdot X(T) - \zeta(0)\cdot X(0) - 
    \int_0^T\!\e^{-\zeta(t)\cdot X(t)} \big( \partial_t + \sQ_t\big)e^{\zeta(t)\cdot X(t)}\dd t.
\label{eq:app change of measure}
\end{equation}
\label{th:change of measure general}
\end{theorem}

\begin{proof}
We only need to show the direction ``$\impliedby$''; the converse then follows immediately from the uniqueness of the generator. To this end define $\widehat \PP_\zeta$ by  \eqref{eq:app change of measure} and let the associated expectation operator be $\widehat \EE_\zeta$. We sketch a number of steps, similar to~\cite[Appendix~1, Sect.~7]{Kipnis1999} and \cite{Palmowski2002}, by which it is shown that under $\widehat \PP_\zeta$ $X$ is Markov with generator $\sQ_{\zeta, t}$.  

\begin{enumerate}[1.]
\item Define for $t\in(0,T)$, the process
\begin{equation*}
  E(t):=\exp\!\left(
    \zeta(t)\cdot X(t) - \zeta(0)\cdot X(0) - 
    \int_0^t\!\e^{-\zeta(s)\cdot X(s)} \big( \partial_s + \sQ_s\big)e^{\zeta(s)\cdot X(s)}\dd s
    \right)
\end{equation*}
and recall $E(T) = \lim_{t\nearrow T}E(t)$.
By Lemma~\ref{lem:Martingale linear test function} above, $E(t)$ is a strictly positive, mean-one $\PP$-Martingale. One then shows that $\left.\frac{\dd \widehat \PP_\zeta}{\dd \PP}\right\vert_{\sF_t} = E(t)$ and $\left.\frac{\dd \PP}{\dd\widehat  \PP_\zeta}\right\vert_{\sF_t} = \frac{1}{E(t)}$.

\item For any $Y \in L^1(\Omega, \sF)$, using the definition of conditional expectation and the results from the previous point, it follows that
$\widehat \EE_\zeta\left[Y\middle\vert\sF_t\right] = \EE\left[Y E(T)/E(t)\middle\vert\sF_t\right]$.

\item Next one can use the result from point 2 to show via conditional expectations under $\PP$ and the $\PP$-Markov property that for $t \geq s$ and any bounded and measurable $f:\X\to\RR$, we have 
$\widehat \EE_\zeta\left[f(X(t)) \middle\vert\sF_s\right] =\widehat  \EE_\zeta\left[f(X(t)) \middle\vert\sigma(X(s))\right]$, and so $X$ is $\widehat \PP_\zeta$-Markov.

\item Finally, the propagators $(P^\zeta_{s,t}f)(x) := \widehat \EE_\zeta\left[f(X(t)) \middle\vert X(s)= x \right]$ then satisfy $\frac{\dd}{\dd s}(P_{s,t}f)(x) = -(\sQ_{\zeta,s} P_{s,t}f)(x)$.  This implies that under $\widehat \PP_\zeta$, $X$ has the same finite dimensional distributions as the process with generator $\sQ_{\zeta, t}$ and thus  $\widehat \PP_\zeta = \PP_\zeta$.
\end{enumerate}
\end{proof}

\section*{Acknowledgements}

This research has been funded by the
Deutsche Forschungsgemeinschaft (DFG) through grant 
CRC 1114 "Scaling Cascades in Complex Systems", 
Project C08.

\bibliographystyle{alpha}
\bibliography{library}

\end{document}